\documentclass[12pt]{article}
\usepackage{a4wide}
\usepackage{amsmath, amsthm, amsfonts, amssymb, bbm}
\usepackage{graphics}
\usepackage{xypic}
\usepackage[all]{xy}
\usepackage[english]{babel}
\usepackage[font=small,format=plain,labelfont=bf,up]{caption}
\usepackage{hyperref}

\allowdisplaybreaks

\theoremstyle{plain}

\newtheorem{theorem}{Theorem}
\newtheorem{lemma}[theorem]{Lemma}
\newtheorem{proposition}[theorem]{Proposition}
\newtheorem{corollary}[theorem]{Corollary}
\newtheorem{conjecture}[theorem]{Conjecture}

\theoremstyle{definition}

\newtheorem{remark}[theorem]{Remark}
\newtheorem{definition}[theorem]{Definition}

 \numberwithin{equation}{section}
 \numberwithin{theorem}{section}

\DeclareMathOperator{\End}{End}
\DeclareMathOperator{\Hom}{Hom}
\DeclareMathOperator{\Rep}{Rep}
\DeclareMathOperator{\ev}{ev}
\DeclareMathOperator{\coev}{coev}
\DeclareMathOperator{\CE}{CE}
\DeclareMathOperator{\CF}{CF}
\DeclareMathOperator{\Gr}{Gr}
\DeclareMathOperator{\Irr}{Irr}

\newcommand{\funF}{\mathcal{F}}

\newcommand{\oneL}{1_{\Lambda}}

\newcommand\be            {\begin{equation}}
\newcommand\ee            {\end{equation}}

\newcommand{\ot}{\otimes}

\newcommand\SF{\mathcal{S}\hspace{-.65pt}\mathcal{F}}

\newcommand{\K}{\kappa}
\newcommand{\tildechi}{\phi}
\newcommand{\Rad}{\rho}

\newcommand{\muldiag}{\hspace{3pt}\begin{picture}(-1,1)(-1,-3)\circle*{4}\end{picture}\hspace{6pt}}

\newcommand\eps           {\varepsilon}
\newcommand\id            {id}
\newcommand\Id            {I\hspace{-1pt}d}
\newcommand\one           {{\bf1}}

\newcommand\Cb            {\mathbb{C}}

\newcommand\Zb            {\mathbb{Z}}

\newcommand\Cc            {\mathcal{C}}
\newcommand\Ec            {\mathcal{E}}
\newcommand\Gc            {\mathcal{G}}

\newcommand\Zc            {\mathcal{Z}}

\newcommand\h            {\mathfrak{h}}
\newcommand\VOA			{\mathcal{V}}

\newcommand{\Vect}{{\cal V}{\it ect}}
\newcommand{\sVect}{{\it s}{\cal V}{\it ect}}

\begin{document}

\thispagestyle{empty}
\def\thefootnote{\fnsymbol{footnote}}
\begin{flushright}
ZMP-HH/16-8\\
Hamburger Beitr\"age zur Mathematik 593
\end{flushright}
\vskip 3em
\begin{center}\LARGE
The non-semisimple Verlinde formula\\ 
and pseudo-trace functions
\end{center}

\vskip 2em
\begin{center}
{\large 
Azat M. Gainutdinov\,$^{a,b,c}$~~and~~Ingo Runkel\,$^a$~~\footnote{Emails: {\tt azat.gainutdinov@uni-hamburg.de}, {\tt ingo.runkel@uni-hamburg.de}}}
\\[1.5em]
{\sl\small $^a$ Fachbereich Mathematik, Universit\"at Hamburg\\
Bundesstra\ss e 55, 20146 Hamburg, Germany}\\[0.5em]
{\sl\small $^b$ DESY, Theory Group,\\ Notkestra\ss e 85, 22603 Hamburg,
Germany}\\[0.5em]
{\sl\small $^c$ Laboratoire de Math\'ematiques et Physique Th\'eorique CNRS,\\
Universit\'e de Tours,
Parc de Grammont, 37200 Tours, 
France}
\end{center}

\vskip 2em
\begin{center}
  May 2016
\end{center}
\vskip 2em

\begin{abstract}
Using results of Shimizu on internal characters we prove a useful non-semisimple variant of the categorical Verlinde formula for factorisable finite tensor categories.
Conjecturally, examples of such categories are given by the representations $\Rep\VOA$ of a vertex operator algebra $\VOA$ subject to certain finiteness conditions.
Combining this with results on pseudo-trace functions by Miyamoto and Arike-Nagatomo, one can make a precise conjecture for a non-semisimple modular Verlinde formula which relates modular properties of pseudo-trace functions for $\VOA$ 
and the product in the Grothendieck ring of $\Rep\VOA$.
We test this conjecture in the example of
the vertex operator algebra of
$N$ pairs of symplectic fermions
by explicitly computing the modular $S$-transformation of the pseudo-trace functions.
\end{abstract}

\setcounter{footnote}{0}
\def\thefootnote{\arabic{footnote}}

\newpage

\tableofcontents

\section{Introduction}

One of the more surprising outcomes of the theory of vertex operator algebras (VOAs) is the Verlinde formula for VOAs with semisimple representation theory (and some additional properties) \cite{Verlinde:1988sn,Huang:2004bn}. It relates the behaviour of characters of the VOA under the modular $S$-transformation with the fusion tensor product of irreducible representations.

\medskip

As explained clearly
in \cite{Fuchs:2006nx}, there are several aspects to the Verlinde formula. We will focus on the categorical aspect and the modular aspect.

\medskip

\noindent
{\em Categorical Verlinde formula:} Let $k$ be an algebraically closed field of characteristic zero and let $\Cc$ be a $k$-linear finite braided tensor category.
Assume further that $\Cc$ is {\em factorisable}, which is a non-degeneracy condition on the braiding (see Section \ref{sec:non-deg}). 

Let now $\Cc$ be in addition semisimple 
-- $\Cc$ is then called a {\em modular tensor category}
(to be precise, `modular' also requires ribbon).
A key application of such categories is that they are the defining data of a certain type of three-dimensional topological field theory (3d\,TFT) \cite{retu,tur,Bartlett:2015baa}. From this point of view, it is not too surprising that certain Hom-spaces of $\Cc$ carry a mapping class group action (actually, maybe less obviously, a projective such action). In particular, for
\be
	L ~:= \bigoplus_{U \in \Irr(\Cc)} U^* \ot U ~~ \in~ \Cc\ ,
\ee
where $\Irr(\Cc)$ is a choice of representatives of the isomorphism classes of simple objects in $\Cc$,
the space $\Cc(\one,L)$ carries a projective $SL(2,\Zb)$-action. 
$\Cc(\one,L)$ has a preferred basis given by the coevaluation maps $\chi_U := \widetilde\coev_U : \one \to U^* \ot U$. The generator $S = \big( \begin{smallmatrix} 0 & -1 \\ 1 & 0 \end{smallmatrix} \big)$ of $SL(2,\Zb)$ is represented by
\be\label{eq:intro-SC}
	S_\Cc(\chi_U) = \big(\mathrm{Dim}\, \Cc \big)^{-\tfrac12} \hspace{-.3em} \sum_{X \in \Irr(\Cc)} \hspace{-.3em} s_{U^*,X}\, \chi_X  
	\qquad , \quad
	\mathrm{Dim}\,\Cc = \hspace{-.3em}\sum_{X \in \Irr(\Cc)} \hspace{-.3em}s_{\one,X}
	 \ ,
\ee
where $s_{A,B} \in k$ is the categorical trace of the double-braiding $c_{B,A} \circ c_{A,B}$. In the application of ribbon categories to link invariants, $s_{A,B}$ is the invariant of the Hopf-link with appropriate orientations and where the two ribbon circles are labelled $A$ and $B$, respectively. 

The Grothendieck ring $\Gr(\Cc)$ of $\Cc$ is freely generated as a $\Zb$-module by the classes $[U]$, $U \in \Irr(\Cc)$. The tensor product induces a ring structure with structure constants
\be
	[U] \cdot [V] = \sum_{W \in \Irr(\Cc)} N_{UV}^{~W} \, [W] 
	 \ .
\ee
	The structure constants of $\Gr(\Cc)$ are also called {\em fusion rules} of $\Cc$.
The categorical Verlinde formula for a modular tensor category $\Cc$ states
\be\label{eq:intro-ver-orig}
	N_{UV}^{~W} ~=~ \frac{1}{\mathrm{Dim} \,\Cc }\sum_{X \in \Irr(\Cc)} \frac{s_{U,X}s_{V,X}s_{W^*,X}}{s_{\one,X}} \ .
\ee
This algebraic result is relatively straightforward to prove, see e.g.\ \cite[Thm.\,4.5.2]{tur}. 

In Section \ref{sec:GrVer} we recall an isomorphism between $\Cc(\one,L)$ and 
the space of endomorphisms of the identity functor $\End(\Id_\Cc)$ \cite{Lyubashenko:1995}.
Transporting the algebra structure of $\End(\Id_\Cc)$ to $\Cc(\one,L)$ results in the product $\chi_U \muldiag \chi_V = \delta_{U,V} (\mathrm{Dim}\Cc)^{\frac12}/s_{\one,U} \cdot \chi_U$.
One can rewrite \eqref{eq:intro-ver-orig} as
\be\label{eq:intro-ver}
	S_\Cc^{-1}\big( \, S_\Cc(\chi_U) \muldiag S_\Cc(\chi_V)\, \big)
	~= \hspace{-.3em}
	\sum_{W \in \Irr(\Cc)} \hspace{-.5em} N_{U,V}^{~W} \, \chi_W \ .
\ee

Results of Shimizu \cite{Shimizu:2015} 
imply that \eqref{eq:intro-ver} still holds if we drop the semisimplicity requirement from $\Cc$ (Theorem \ref{thm:ver}). All that changes is that $L$ is generalised to the coend $L = \int^{X \in \Cc} X^* \ot X$, and that the action of $S$ on $\Cc(\one,L)$ is defined in terms of a non-degenerate Hopf-pairing and an integral on $L$ \cite{Lyubashenko:1995}, see Section \ref{sec:GrVer}.

\medskip

\noindent
{\em Modular Verlinde formula:} 
Let $\VOA$ be a VOA which is $C_2$-cofinite, simple as a $\VOA$-module, isomorphic to the contragredient module $\VOA'$, and non-negatively graded with $\VOA_0 = \Cb \one$. For $\VOA$ one can define a space of torus one-point functions $C_1(\VOA)$, which is a subspace of the space of functions $\VOA \times \mathbb{H} \to \Cb$ which are linear in $\VOA$ and holomorphic on the upper half plane $\mathbb{H}$. $C_1(\VOA)$ is finite dimensional and invariant under modular transformations \cite{Zhu1996}.

If $\VOA$ is in addition rational (and so $\Rep(\VOA)$ is semisimple), Zhu proved that $C_1(\VOA)$ has a basis given by the characters $\hat\chi_U$ of irreducible $\VOA$-modules. Consequently, there is a unique matrix $\mathbb{S}_{U,X}$ such that the image $\hat\chi_U|_S$ of the character $\hat\chi_U$ under the modular $S$-transformation is
\be\label{eq:ssi-irred-char-S}
	\hat\chi_U|_S ~= 
	\hspace{-1em}
	\sum_{X \in \Irr(\Rep(\VOA))} \hspace{-1em} \mathbb{S}_{U,X} \,  
	\, \hat\chi_X \ .
\ee 
In the modular Verlinde formula, one defines the map $S_\VOA$ in the same way as $S_\Cc$ in \eqref{eq:intro-SC}, but with $\mathbb{S}_{U^*,X}$ instead of $(\mathrm{Dim} \Cc)^{-1/2} \, s_{U^*,X}$. The claim then is that \eqref{eq:intro-ver}, with $S_\VOA$ in place of $S_\Cc$, computes the structure constants of the Grothendieck ring $\Gr(\Rep(\VOA))$.
To prove the modular Verlinde formula in this generality (rational VOAs with the above properties) is hard and was achieved by Huang in \cite{Huang:2004bn}.

\medskip

Conceptually, the relation between the categorical and modular Verlinde formula can be understood as follows. $\Rep(\VOA)$ is a modular tensor category and the topological modular functor defined by the 3d\,TFT obtained from $\Rep(\VOA)$ should in a suitable sense be equivalent to the complex modular functor obtained from the conformal blocks associated to $\VOA$. While in this generality, the equivalence of modular functors is an open problem, it is known that the categorical and the modular $SL(2,\Zb)$-actions
agree \cite{Huang2005}.

\medskip

Let us now drop the rationality requirement, so that $\Rep(\VOA)$ is not necessarily semisimple (but still finite abelian, see Section \ref{sec:voa}). A version of 3d\,TFTs for not necessarily semisimple factorisable finite ribbon categories was given in \cite{Kerler:2001}. 
Combining this with results on internal characters \cite{Shimizu:2015} (see Section \ref{sec:GrVer}) and with pseudo-trace functions  \cite{Miyamoto:2002ar,Arike:2011ab} (see Section \ref{sec:ptfun}) leads to a precise conjecture on how the modular Verlinde formula might generalise to such non-semisimple situations. This is detailed in Theorem \ref{thm:main} (which is based on conjectures given there).

The generalisation states that -- conjecturally -- the product in $\Gr(\Rep(\VOA))$ can be computed by the following procedure. (It is not a priori clear why the maps appearing below exist. This is part of the conjecture, see Section \ref{sec:voa} for details.)
\begin{enumerate}
\item Compute the irreducible $\VOA$-modules $U \in \Irr(\Rep(\VOA))$ and choose a projective generator $G$ of $\Rep(\VOA)$.

\item The pseudo-trace functions for $G$ are parametrised by central forms $C(\End_\VOA(G))$ on $\End_\VOA(G)$.
Denote by $\varphi_U \in C(\End_\VOA(G))$, $U \in\Irr(\Rep(\VOA))$, the
elements corresponding to the irreducible characters $\hat\chi_U$. The $\varphi_U$ are expressed as traces over irreducible $\End_\VOA(G)$-modules in \eqref{eq:varphiM-via-tr} below.
Compute the endomorphism $S_\VOA$
-- given by the modular $S$-transformation -- on the linear span of the $\varphi_U$.

\item The modular transformation of the vacuum character can be used to define  a linear isomorphism between $\End(\Id_{\Rep(\VOA)})$ and $C(\End_\VOA(G))$. Use this isomorphism to transport the algebra structure of $\End(\Id_{\Rep(\VOA)})$ to $C(\End_\VOA(G))$ and denote the resulting product 
	by `$\muldiag$'.
\item For $A,B \in \Irr(\Rep(\VOA))$ compute the unique constants $N_{AB}^{~C}$ in
\be
	S_\VOA^{-1}\big( \, S_\VOA(\varphi_A) \muldiag S_\VOA(\varphi_B) \, \big)
	~= \hspace{-1em}
	\sum_{C \in \Irr(\Rep(\VOA))} \hspace{-1em} N_{AB}^{~C} ~ \varphi_C \ .
\ee
\end{enumerate}
Conjecturally, the $N_{AB}^{~C}$ are the structure constants of $\Gr(\Rep(\VOA))$.

\medskip

In Section \ref{sec:sf} we apply this procedure to the VOA $\VOA_\mathrm{ev}$ given by the even part of $N$ pairs of symplectic fermions. The result agrees with the (also conjectural) fusion rules of $\VOA_\mathrm{ev}$.

\medskip

Finitely non-semisimple generalisations of the Verlinde formula 
were first considered in \cite{Flohr:2001zs} and in \cite{Fuchs:2003yu}. 
Further previous discussions can be found in 
\cite{Fuchs:2006nx,Flohr:2007jy,Gaberdiel:2007jv,Gainutdinov:2007tc,Pearce:2009pg}, 
see Remark \ref{rem:comments-ver} for more details.

\medskip

Related results on the logarithmic Verlinde formula and pseudo-trace functions have been obtained independently in \cite{Creutzig:2016fms}.

\bigskip

\noindent
{\bf Acknowledgements:}
We thank Thomas Creutzig, Vanda Farsad,  Terry Gannon,  Ehud Meir  and Christoph Schweigert for fruitful discussions.
We are grateful to
	Christian Blanchet, Vanda Farsad, J\"urgen Fuchs, Antun Milas and David Ridout
for helpful comments on a draft of this paper.
The work of AMG was  supported by a Humboldt fellowship,  DESY and CNRS.
The authors would like to thank the organisers of the Humboldt Kolleg ``Colloquium on Algebras and Representations -- Quantum 2016'' (Cordoba, Argentina), where this work was started, for a stimulating workshop and for their hospitality.

\bigskip

In all of this paper, $k$ will be an algebraically closed field of characteristic zero. 

\section{Non-degeneracy for braided finite tensor categories}
\label{sec:non-deg}

In this section we briefly recall Shimizu's result on the equivalence of several non-degeneracy conditions for braided finite tensor categories \cite{Shimizu:2016}.

\medskip

Following \cite{Etingof:2003}, a {\em finite tensor category} $\Cc$ is a category which
\begin{itemize}
\item is a $k$-linear finite abelian category (i.e.\ equivalent as a $k$-linear category to that of finite-dimensional modules over a finite-dimensional $k$-algebra),
\item is a monoidal category with $k$-bilinear tensor product functor,
\item is rigid (i.e.\ has left and right duals),
\item has a simple tensor unit $\one$.
\end{itemize}
We remark that, because $\Cc$ is rigid, the tensor product functor is exact in each argument.

\medskip

Let $\Cc$ be a braided finite tensor category, with braiding $c$ and 
left  (co)evaluation maps $\ev_X: X^*\ot X \to \one $ and $\coev_X: \one \to X\ot X^*$, for $X\in \Cc$ and its left dual $X^*$.
 Let $L$ be the coend 
(see e.g.\ \cite{Kerler:2001} or the review 
in~\cite[Sec.~4]{Fuchs:2010mw})
\be
	L ~=\, \int^{X \in \Cc} \hspace{-.5em} X^* \ot X \ .
\ee
We denote by 
\be
	\iota_X : X^* \ot X \longrightarrow L \quad , \quad  X \in \Cc \ ,
\ee
the corresponding dinatural transformation.
The coend $L$ exists since $\Cc$ is a finite tensor category (see e.g.\ \cite[Cor.\,5.1.8]{Kerler:2001}). It carries the structure of a Hopf algebra
in $\Cc$ and is equipped with a Hopf pairing $\omega : L \ot L \to \one$ \cite{Lyubashenko:1995}. The Hopf pairing is defined uniquely via (we omit the `$\ot$' between objects
	and use `$\sim$' for a successive application of associativity or unit isomorphisms)
\begin{align}
	\omega \circ (\iota_X \ot \iota_Y)
~=~& \big[
(X^*  X)  (Y^*  Y)
\xrightarrow{\sim}
X^*  ((X  Y^*)  Y)
\xrightarrow{\id \ot (c_{Y^*,X} \circ c_{X,Y^*}) \ot \id}
X^*  ((X  Y^*)  Y)
\nonumber\\
& \quad
\xrightarrow{\sim}
(X^*  X)  (Y^*  Y)
\xrightarrow{
\ev_X \ot \ev_Y}
\one  \one
\xrightarrow{\sim}
\one \big] \ .
\end{align}

As in \cite{Lyubashenko:1995,Shimizu:2015}, we define the space of {\em central elements} as 
$\CE(\Cc) = \Cc(L,\one)$
 and that of {\em class functions} as $\CF(\Cc) = \Cc(\one,L)$. 
The Hopf pairing of $L$ defines the linear map
\be\label{eq:Omega-def}
	\Omega : \CF(\Cc) \longrightarrow \CE(\Cc) \quad , \quad f \mapsto \omega \circ (f \ot \id)  \ . 
\ee

With these preparations, we can state four natural non-degeneracy conditions on the braided finite tensor category $\Cc$:
\begin{enumerate}
\item Every transparent object in $\Cc$ is isomorphic to a direct sum of tensor units. ($T \in\Cc$ is {\em transparent} if for all $X \in \Cc$, $c_{X,T} \circ c_{T,X} = \id_{T\ot X}$.)
\item The canonical braided monoidal functor 
$\Cc \boxtimes \overline{\Cc} \to \Zc(\Cc)$ 
is an equivalence. (Here, $\boxtimes$ is the Deligne product, $\overline{\Cc}$ is the same tensor category as $\Cc$, but has inverse braiding, 
and $\Zc(\Cc)$ is the Drinfeld centre of $\Cc$.)
\item The pairing $\omega : L \ot L \to \one$ is non-degenerate (in the sense that there exists a copairing $\one \to L \ot L$).
\item $\Omega$ is an isomorphism.
\end{enumerate}
For $\Cc$ semisimple, it has been known for some time that conditions
 1--4 are 
equivalent \cite{Bruguieres:2000,Muger2001b}. Recently, Shimizu was able to show this equivalence in general:

\begin{theorem}[\cite{Shimizu:2016}]\label{thm:factequiv}
For a braided finite tensor category $\Cc$,
conditions 1--4 are equivalent.
\end{theorem}

Following the nomenclature in Hopf algebras, we say:

\begin{definition}
A braided finite tensor category satisfying the equivalent conditions 
1--4 above is called {\em factorisable}.
\end{definition}

For a factorisable finite tensor category $\Cc$, the coend $L$ admits a left/right integral\footnote{A left integral is a morphism $\Lambda_L:\one \to L$ satisfying $\mu_L\circ(\id_L\otimes\Lambda_L)=\Lambda_L\circ \varepsilon_L$, where $\mu_L$ is the multiplication on $L$ and $\varepsilon_L$ is the counit of $L$. A right integral is defined similarly.} 
$\Lambda_L : \one \to L$
	satisfying $\omega \circ (\Lambda_L \ot \Lambda_L) = \id_\one$ \cite{Lyubashenko:1995}, see also \cite[Sect.\,5.2.3]{Kerler:2001}.
	An integral is unique up to a scalar factor, and so the normalisation condition $\omega \circ (\Lambda_L \ot \Lambda_L) = \id_\one$ determines $\Lambda_L$ up to a sign. 
The existence of $\Lambda_L$ also implies that $\Cc$ is unimodular. (See \cite{Etingof:2004} for the definition of unimodularity; for the relation to integrals, see \cite[Thm.\,6.8]{Shimizu:2014}).

\section{A categorical Verlinde formula}\label{sec:GrVer}

In this section, $\Cc$ denotes a braided finite tensor category over $k$ which is in addition pivotal. Write 
 $\delta_X : X \to X^{**}$ for the pivotal structure of $\Cc$. 
We choose the right dual of $X\in\Cc$ to be equal to the left dual $X^*$ and define the right duality morphisms as
\begin{align}
\widetilde\ev_X ~&=~ \big[\,X \ot X^* \xrightarrow{\delta_X \ot \id} X^{**} \ot X^* \xrightarrow{\ev_{X^*}} \one \,\big] \ ,
\nonumber \\
\widetilde\coev_X ~&=~ \big[\, \one \xrightarrow{\coev_{X^*}}
 X^{*} \ot X^{**}
\xrightarrow{\id \ot \delta_X^{-1}} X^{*} \ot X \,\big] \ .
\end{align}

\medskip

Denote by $\Irr(\Cc)$ a choice of representatives of the isomorphism classes of simple objects in $\Cc$. Since $\Cc$ is finite, $\Irr(\Cc)$ is a finite set.
The Grothendieck ring $\Gr(\Cc)$ of $\Cc$ is freely generated as a $\Zb$-module by the classes $[U]$, $U \in \Irr(\Cc)$. We write $\Gr_k(\Cc) := k 
	\otimes_\Zb
\Gr(\Cc)$ for the corresponding $k$-algebra. The structure constants $N_{U,V}^{~W}$ of $\Gr(\Cc)$ are defined via
\be
	[U] \, [V] ~=~ \sum_{W \in \Irr(\Cc)} N_{U,V}^{~W} \, [W] 
	\qquad , ~~ U,V \in \Irr(\Cc) \ .
\ee

\medskip

The {\em internal character}	of $V \in \Cc$ is the element $\chi_V \in \CF(\Cc)$ given by \cite{Fuchs:2013lda,Shimizu:2015}\footnote{
  We deviate slightly from the convention in \cite[Sect.\,5.2]{Shimizu:2016}, where the internal character is defined as $\chi^\text{Shimizu}_{V} = \iota_{V^*} \circ (\delta_V \ot \id) \circ \coev_V$.  The relation to \eqref{eq:chiV-def} is $\chi_V = \chi^\text{Shimizu}_{V^*}$. In the semisimple case and for $V \in \Irr(\Cc)$, $\chi_V$  is related to the natural endomorphism $I_V$ of the identity functor given by
   $(I_V)_X = \delta_{V,X} \id_X$ ($X \in \Irr(\Cc)$), while $\chi^\text{Shimizu}_{V}$ is related to $I_{V^*}$, see 
   point 3 of Remark \ref{rem:ver}.
  }
\be\label{eq:chiV-def}
	\chi_V ~=~ 
	\big[\, \one \xrightarrow{\widetilde\coev_V} V^* \ot V   \xrightarrow{\iota_{V}} L \,\big] \ .
\ee

\begin{remark}
When $\Cc$ is $\Rep H$ for $H$ a 
	finite-dimensional
ribbon Hopf algebra over $k$, then $\one=k$
	and $L = H^*$ with coadjoint action. Thus
the images $\chi_V(1)$
are linear forms on~$H$ 
invariant under the coadjoint action (called $q$-characters). They are equal to the trace functions $\chi_V(1)=\mathrm{Tr}_V(uv^{-1}\,\cdot\,)$, with the ribbon element $v$ and the Drinfeld element~$u$.
\end{remark}

\begin{theorem}[{\cite[Cor.\,4.2\,\&\,Thm.\,5.12]{Shimizu:2015}}] 
\label{thm:chiinj}
The map $V \mapsto \chi_V$ factors through $\Gr(\Cc)$. The induced map $\chi : \Gr_k(\Cc) \to \CF(\Cc)$, $[V] \mapsto \chi_V$, is injective. Suppose in addition that $\Cc$ is unimodular. Then $\chi$ is surjective iff $\Cc$ is semisimple.
\end{theorem}

This theorem holds more generally in pivotal finite tensor categories. If $\Cc$ is factorisable, it is automatically unimodular, so the last part of the theorem applies in this case.

\medskip

The Hopf algebra structure on $L$ induces an algebra structure on central elements and on class functions: Let $f,g \in \CE(\Cc)$ and $a,b \in \CF(\Cc)$. Then
\begin{align}
	f * g &\,:=\, \big[ L \xrightarrow{\Delta_L} L \ot L \xrightarrow{f \ot g} \one \ot \one \xrightarrow{\sim} \one \big] \ ,
	\nonumber \\
	a \cdot b &\,:=\, \big[\, \one \xrightarrow{\sim} \one \ot \one \xrightarrow{ a \ot b} L \ot L \xrightarrow{\mu_L} L \,\big] \ .
\end{align}
The units are given by the counit and unit of $L$, respectively: $1_{\CE} = \eps_L$, $1_{\CF} = \eta_L$.

Let $\End(\Id_\Cc)$ denote the $k$-algebra of natural endomorphisms of the identity functor on $\Cc$. Given $\alpha \in \End(\Id_\Cc)$, we obtain an element 
$\psi(\alpha) \in \CE(\Cc)$ via $\psi(\alpha) \circ \iota_X = \ev_X \circ (\id \ot \alpha_X)$ for all $X \in \Cc$. This defines a $k$-linear map 
\be
\psi : \End(\Id_\Cc) \longrightarrow \CE(\Cc) \ .
\ee
We have ({\cite{Lyubashenko:1995}, see also {\cite[Prop.\,5.2.5]{Kerler:2001}):

\begin{lemma}\label{lem:EndId-CE}
$\psi$ is an isomorphism of $k$-algebras.
\end{lemma}

In particular, $\CE(\Cc)$ is commutative. The inverse map to $\psi$ can be given explicitly: For $X \in \Cc$, $f \in \CE(\Cc)$ (omitting `$\ot$' between objects)
\be
	\psi^{-1}(f)_X = \big[ X \xrightarrow{\sim} \one X
	\xrightarrow{\coev_X \ot \id} (X  X^*)  X
	\xrightarrow{\sim} X  (X^*  X)
	\xrightarrow{\id \ot \iota_X} X L
	\xrightarrow{\id \ot f} X \one
	\xrightarrow{\sim} X\big] \ .
\ee

Recall the map $\Omega : \CF(\Cc) \to \CE(\Cc)$ defined in \eqref{eq:Omega-def}.

\begin{lemma} \label{lem:Omalg}
$\Omega$ is a $k$-algebra map.
\end{lemma}

\begin{proof}
This is an immediate consequence of $\omega$ being a Hopf pairing (see e.g.\ \cite[Sect.\,3.1]{Shimizu:2016} for conventions) and of $\CE(\Cc)$ being commutative.
\end{proof}

\begin{remark}
When $\Cc=\Rep H$ for a finite-dimensional 
quasi-triangular
 Hopf algebra~$H$, 
the map $\Omega$ is the Drinfeld mapping from the space of $q$-characters (linear forms on $H$ invariant under the coadjoint $H$-action) to the centre of $H$, given by $\chi(\cdot)\mapsto \chi(M')M''$ for the monodromy matrix $M=R_{21}R_{12}$, while Lemma~\ref{lem:Omalg} reduces to Drinfeld's lemma~\cite{Drinfeld}. For a factorisable Hopf algebra, the Drinfeld mapping is an isomorphism of the algebras.
\end{remark}

Recall from Theorem \ref{thm:factequiv} that the invertibility of $\Omega$ was one of the equivalent characterisations of factorisability
of a braided finite tensor category $\Cc$.

\begin{lemma}[{\cite[Sec.\,4.5]{Fuchs:2010mw} and \cite[Thm.\,3.11\,\&\,Prop.\,3.14]{Shimizu:2015}}]
$\chi : \Gr_k(\Cc) \to \CF(\Cc)$ is a $k$-algebra map.
\end{lemma}

By Theorem \ref{thm:chiinj}, the internal 
characters $\{ \chi_U \,|\, U \in \Irr(\Cc) \}$ are linearly independent in $\CF(\Cc)$. Since $\chi$ is an algebra map,
i.e.\ $\chi_{V\otimes W} = \chi_V \cdot \chi_W$,
 the structure constants of $\Gr_k(\Cc)$ can be uniquely recovered from
\be
\chi_U \cdot \chi_V ~=~ \sum_{W \in \Irr(\Cc)} N_{U,V}^{~W} \, \chi_W \ .
\ee

As a corollary to Theorem \ref{thm:chiinj} and Lemma \ref{lem:Omalg}, one obtains:

\begin{corollary}\label{cor:ver}
For $U,V \in \Irr(\Cc)$,
\be
	\Omega^{-1}\big( \, \Omega(\chi_U) * \Omega(\chi_V)\, \big)
	~=~
	\sum_{W \in \Irr(\Cc)} N_{U,V}^{~W} \, \chi_W \ .
\ee
\end{corollary}

To get a variant of the Verlinde formula, our next point is to introduce the modular $S$-transformation on $\End(\Id_\Cc)$. 
	To do so, we will need the isomorphism 
\be\label{eq:R-map-def}
	\Rad : \CE(\Cc) \longrightarrow \CF(\Cc)
	\quad , \quad
	f \longmapsto (f \ot \id) \circ \Delta_L \circ \Lambda_L \ .
\ee
	The inverse of $\Rad$ can be given explicitly in terms of the cointegral $\lambda_L : L \to \one$ of $L$ obtained 
from $\Lambda_L$ via $\lambda_L = \Omega(\Lambda_L)$. 
Note that from the normalisation condition on $\Lambda_L$, we have $\lambda_L \circ \Lambda_L = \id_\one$. For the inverse of $\Rad$ one finds
\be
\Rad^{-1}(a) = \big[\, L \xrightarrow{\sim} \one L \xrightarrow{a \ot \id} LL
\xrightarrow{S_L \ot \id} LL \xrightarrow{\mu_L} L \xrightarrow{\lambda_L} \one \,\big] \ ,
\ee
see e.g.\ \cite[Cor.\,4.2.13]{Kerler:2001}. We remark that $\Rad$ is in general not an algebra map. This can be seen explicitly in the example in Section \ref{sec:sf}.

\medskip

We define the map $S_\Cc : \End(\Id_\Cc) \to \End(\Id_\Cc)$ as
\be\label{eq:SC-def}
S_\Cc ~=~ \big[\, \End(\Id_\Cc) \xrightarrow{\,\psi\,} \CE(\Cc) \xrightarrow{~\Rad~} \CF(\Cc) \xrightarrow{~\Omega~} \CE(\Cc) \xrightarrow{~\psi^{-1}~} \End(\Id_\Cc) \, \big] \ .
\ee
We have seen that $\psi$ and $\Omega$ are algebra maps, while $\Rad$ is in general not.
Thus in general $S_\Cc$ is not an algebra map, either.

\begin{remark}
In the application to mapping class groups developed in \cite{Lyubashenko:1995,Lyubashenko:1994tm}, 
$\CF(\Cc) = \Cc(\one,L)$ is the vector space associated to a torus and carries a projective $SL(2,\Zb)$ action. The operator implementing the modular $S$-transformation on $\CF(\Cc)$ is given by 
$S_{\CF} := \Rad \circ \Omega : \CF(\Cc) \to \CF(\Cc)$, which, when transported to $\End(\Id_\Cc)$ via 
$(\Rad \circ \psi)^{-1}$, results in $S_\Cc$. 
\end{remark}

Define, 	for $V \in \Cc$,
\be\label{eq:tildechi-def}
\tildechi_V := \psi^{-1}(\Rad^{-1}(\chi_V)) \in \End(\Id_\Cc)\ .
\ee
Since $\psi$ and $\Rad$ are isomorphisms, by Theorem~\ref{thm:chiinj} the set $\{\, \tildechi_U \,|\, U \in \Irr(\Cc)\,\}$ is linearly independent.
We can now restate Corollary~\ref{cor:ver} as the following theorem, which is the Verlinde-type formula we wish to employ later:

\begin{theorem}\label{thm:ver}
Let $\Cc$ be a factorisable pivotal finite tensor category. Then the structure constants of $\Gr(\Cc)$ can be uniquely recovered from knowing
\begin{itemize}
\item the algebra $\End(\Id_\Cc)$,
\item the elements $\tildechi_U \in \End(\Id_\Cc)$ for $U \in \Irr(\Cc)$,
\item the linear map $S_\Cc : \End(\Id_\Cc) \to \End(\Id_\Cc)$,
\end{itemize}
via, for $U,V \in \Irr(\Cc)$,
\be\label{eq:ver}
	S_\Cc^{-1}\big( \, S_\Cc(\tildechi_U) \circ S_\Cc(\tildechi_V)\, \big)
	~=~
	\sum_{W \in \Irr(\Cc)} N_{U,V}^{~W} \, \tildechi_W \ .
\ee
\end{theorem}

We state the theorem in this form, because it matches the data used in the conjectural modular Verlinde formula in Section \ref{sec:voa}.

\begin{remark}\label{rem:ver}
	Let $\Cc$ be as in Theorem \ref{thm:ver}.
\begin{enumerate}
\setlength{\leftskip}{-1em}
\item Equation~\eqref{eq:ver} shows in particular that the linear span of the $S_\Cc(\phi_U)$, $U \in \Irr(\Cc)$, is a 
subalgebra of $\End(\Id_\Cc)$. Thus, to evaluate \eqref{eq:ver} it is enough to know $S_\Cc$ on the linear span of the $\phi_U$. By Theorem \ref{thm:chiinj}, in the non-semisimple case this span is a proper subspace of $\End(\Id_\Cc)$.

\item Substituting the definitions, one checks that the natural transformations 
$S_\Cc(\tildechi_V) = \psi^{-1}(\Omega(\chi_V))$ 
are given by
\begin{align}
 S_\Cc(\tildechi_V)_X
=
\big[
& X 
\xrightarrow{\sim} \one X
\xrightarrow{\widetilde\coev_{V} \ot \id} (V^{*} V) X
\xrightarrow{\sim} V^{*} (V X)
\xrightarrow{\id \ot (c^{-1}_{V,X} \circ c^{-1}_{X,V})} V^{*} (V X)
\nonumber \\ 
& \quad \xrightarrow{\sim} (V^{*} V) X
\xrightarrow{\ev_{V} \otimes \id} \one X
\xrightarrow{\sim}X
\big] \ .
\label{eq:SC-on-tilde-chi}
\end{align}
In particular, $S_\Cc(\tildechi_\one)_X = \id_X$ for all $X$.
Note that $\chi_{\one}$ is the unit $\eta_L$ in $L$ and thus $\psi(\tildechi_{\one}) = \Rad^{-1}(\chi_{\one}) = \lambda_L$ is the cointegral of $L$. Therefore, the operator implementing the modular $S$-transformation on $\CE(\Cc)$ given by $S_{\CE}:= \Omega \circ \Rad: \CE(\Cc) \to \CE(\Cc)$ maps $\lambda_L$ to the unit in $\CE(\Cc)$ which is $\varepsilon_L$. We also have $S_{\CE}(\varepsilon_L) = \lambda_L$ (indeed, $\Rad(\varepsilon_L)=\Lambda_L$ and $\Omega(\Lambda_L)=\lambda_L$). Thus $S_{\CE}^2$ applied to the unit or the cointegral is the identity.
We note that in general the square of the $S$-transformation 
$S_{\CF}=\Rad\circ\Omega$ 
 acts as $\chi_V \mapsto \chi_{V^*}$.

\item
Suppose that $\Cc$ is 
semisimple (in case $\Cc$ is ribbon, this means that $\Cc$ is a modular tensor category).
Then $L = \bigoplus_{U \in \Irr(\Cc)} U^* \ot U$
 and we denote the embedding and projection morphisms by $e_{U}: U^*\otimes U\to L$ and $p_U: L\to U^*\otimes U$. 
The internal characters are then given by $\chi_U = e_U\circ  \widetilde\coev_U$.
 The  integral $\Lambda_L$ and cointegral $\lambda_L = \Omega(\Lambda_L)$ are found to be
\be\label{eq:ssi-LamL}
	\Lambda_L = (\mathrm{Dim}\,\Cc)^{-\frac12} \!\! \sum_{U \in \Irr(\Cc)} \!\!\!\dim(U) \, e_{U} \circ \widetilde\coev_U
	\quad , \quad
	\lambda_L = (\mathrm{Dim}\,\Cc)^{\frac12} \ev_\one \circ\,  p_{\one} \ .
\ee
The normalisation condition $\omega \circ (\Lambda_L \ot \Lambda_L) = \id_\one$ determines $\Lambda_L$ up to a sign. This amounts to fixing a choice of square root $(\mathrm{Dim}\,\Cc)^{1/2}$. The maps $\psi^{-1} : \CE(\Cc)\to\End(\Id_\Cc)$ and $\Rad : \CE(\Cc) \to \CF(\Cc)$ are given by, for $U,X \in \Irr(\Cc)$,
\be
\psi^{-1}(\ev_U \circ p_U)_X = \delta_{U,X} \, \id_X
\quad , \quad
\Rad(\ev_U \circ p_U) = \frac{\dim(U)}{(\mathrm{Dim}\,\Cc)^{\frac12}} \, e_U \circ \widetilde\coev_U \ .
\ee
Using this, one finds that the elements $\tildechi_U \in \End(\Id_\Cc)$ are, for $U,X \in \Irr(\Cc)$,
\be\label{eq:ssi-tildechi}
(\tildechi_U)_X = \delta_{U,X} \, \frac{(\mathrm{Dim}\,\Cc)^{\frac12}}{\dim(U)} \, \id_X  \ .
\ee
The action of $S_\Cc$ from \eqref{eq:SC-on-tilde-chi} becomes, for $U,X \in \Irr(\Cc)$,
\be\label{eq:ssi-S-action}
 S_\Cc(\tildechi_U)_X
	= \frac{s_{U^*,X}}{s_{\one,X}} \, \id_X \ ,
\ee
where the $s_{A,B} \in k$ are determined by the categorical trace via $s_{A,B} \, \id_\one = \mathrm{tr}(c_{B,A} \circ c_{A,B})$. From this, \eqref{eq:ver} produces the usual semisimple categorical Verlinde formula, see e.g.\ \cite[Thm.\,4.5.2]{tur}.

\item
Let $G$ be a projective generator in $\Cc$, for example $G = \bigoplus_{U \in \Irr(\Cc)} P_U$, with $P_U$ the projective cover of $U$. Then $\End(\Id_\Cc) \to Z(\End(G))$, $\eta \mapsto \eta_G$, is an isomorphism of $k$-algebras. 
We can group the $P_U$ according to the block of $\Cc$ they belong to. (By a block we mean the full subcategory given by an equivalence class of objects with respect to the equivalence relation generated by $X \sim Y$ if there is a nonzero morphism $X \to Y$.) Then $\End(\Id_\Cc) \cong Z(\End(G))$ has the block-decomposition
\be
	\End(\Id_\Cc) \cong \bigoplus_{\beta \in \text{blocks}} Z(\End(G_\beta))
	\qquad , ~~\text{where}~~
	G_\beta ~= \hspace{-1.5em}\bigoplus_{U\in\Irr(\Cc) \text{ is in block $\beta$}} \hspace{-1.5em} P_U \ ,
\ee
as $k$-algebras. In this sense, $S_\Cc$ block-diagonalises 
the fusion rules (that is, the structure constants of $\Gr(\Cc)$). 
In particular, in the semisimple case, $P_U=U$ and the fusion rules are fully diagonalised.

\item
	In the non-semisimple case, the idea that the $S$-transformation 
block-diagonalises the fusion rules was explored further in~\cite[Thm.\,4.1.1]{Gainutdinov:2007tc} where a Verlinde formula similar to~\eqref{eq:ver} was also stated. The Verlinde-type formula~\cite[Eqn.\,(4.4)]{Gainutdinov:2007tc} expresses the structure constants $N_{U,V}^W$ in terms of  structure constants for multiplication in the block-diagonalised basis, which contains the $\tildechi_V$ as basis elements. 
Using $S^2_\Cc(\id_{\Cc})=\id_{\Cc}$,
it was also shown that the structure constants for multiplication in the block-diagonalised basis
satisfy a system of linear equations with coefficients from the ``vacuum'' row of the $S$-matrix (the coefficients in the $S$-transformation of the unit), see \cite[Eqn.\,(4.5)]{Gainutdinov:2007tc} for more details. 

\item A result related to Theorem \ref{thm:ver} exists for factorisable ribbon Hopf algebras \cite[Thm.\,3.14]{Cohen:2008}. There, the structure constants are restricted to a subspace (the Higman ideal) and then diagonalised. In the non-semisimple case, the Higman ideal is strictly smaller than the Reynolds ideal
 (which is dual to the space spanned by characters of irreducible representations) \cite[Cor.\,2.3]{Cohen:2008} and the diagonalised form does not allow one to recover the structure constants.
\end{enumerate}
\end{remark}

\section{Pseudo-trace functions for modules of a VOA}\label{sec:ptfun}

We briefly review torus one-point functions of a VOA and their modular properties \cite{Zhu1996,Miyamoto:2002ar}, as well as the construction of such one-point functions via pseudo-traces by Arike and Nagatomo \cite{Arike:2011ab}.

\medskip

Let $\VOA = (\VOA,Y,1,\omega)$ be a VOA. In the seminal paper \cite{Zhu1996}, Zhu starts by defining a functor $\VOA \mapsto A(\VOA)$ from VOAs to $\Cb$-algebras. The algebra $A(\VOA)$ is now called Zhu's algebra. Zhu proved that there is a one-to-one correspondence between irreducible $\VOA$-modules\footnote{
We will not give many details on VOAs and their modules. We just mention that in this section ``module'' stands for ``admissible module'' aka ``$\Zb_{\ge 0}$-gradable weak module''.}
} and irreducible $A(\VOA)$-modules. Furthermore, in \cite{Zhu1996} an important finiteness condition on $\VOA$ is introduced, now called $C_2$-cofiniteness (see e.g.\ \cite{Miyamoto:2002ar} for more details and references). For a $C_2$-cofinite VOA $\VOA$, $A(\VOA)$ is finite-dimensional \cite[Prop.\,3.6]{Dong:1997ea}. Consequently, $\VOA$ only has finitely many simple modules.

\medskip

For a $C_2$-cofinite VOA $\VOA$, Zhu introduces the $\Cb$-vector space of torus 1-point functions $C_1(\VOA)$ (see \cite{Zhu1996,Dong:1997ea} and e.g.\ \cite[Sect.\,5]{Miyamoto:2002ar}). An element $\xi \in C_1(\VOA)$ is a function
\be 
\xi ~:~ \VOA \times \mathbb{H} \longrightarrow \Cb \ ,
\ee
which is linear in $\VOA$, analytic on the upper half plane $\mathbb{H}$, and subject to further conditions for which we refer to \cite{Zhu1996,Dong:1997ea}.
The space $C_1(\VOA)$ is finite-dimensional and invariant under modular transformation in the following sense: Let $\gamma = \big( \begin{smallmatrix} a & b \\ c & d \end{smallmatrix} \big) \in SL(2,\Zb)$ and $\xi \in C_1(\VOA)$. For $v \in \VOA_{[h]}$ (see \cite{Zhu1996,Dong:1997ea} for this grading on $\VOA$) define
\be\label{eq:modxfer}
   \xi|_\gamma(v,\tau) ~:=~
  (c\tau + d)^{-h} \,\xi\big( v , \tfrac{a\tau + b}{c \tau +d} \big)
\ee
and extend linearly. Then (\cite{Zhu1996}, see \cite[Thm.\,5.4]{Dong:1997ea} for this version):

\begin{theorem}\label{thm:C1-mod-inv}
$\xi|_\gamma \in C_1(\VOA)$.
\end{theorem}

\medskip

A VOA $\VOA$ is called {\em rational} if every $\VOA$-module is fully reducible. In analogy to semisimple rings, this already implies that $\VOA$ only has finitely many distinct irreducible modules \cite{Dong:1995}. Furthermore, in this case $A(\VOA)$ is finite-dimensional and semisimple.

For $\VOA$ $C_2$-cofinite and rational, Zhu \cite{Zhu1996} proved the remarkable result that the irreducible characters of $\VOA$ (with insertion of a zero-mode of $\VOA$) span $C_1(\VOA)$ and hence provide vector valued modular forms.
Miyamoto was able to generalise this considerably by dropping the rationality requirement: For a non-negatively graded $C_2$-cofinite VOA $\VOA$ for which every simple $\VOA$-module is infinite-dimensional (cf.\ \cite[Rem.\,3.3.5]{Arike:2011ab}), $C_1(\VOA)$ is spanned by so-called pseudo-trace functions \cite[Thm.\,5.5]{Miyamoto:2002ar}.

\medskip

The pseudo-trace functions in \cite{Miyamoto:2002ar} are defined in terms of $n$'th Zhu algebras. An easier version of pseudo-trace functions was introduced by Arike and Nagatomo \cite{Arike:2011ab}. It is proved in \cite[Thm.\,4.3.4]{Arike:2011ab} that these pseudo-trace functions lie in $C_1(\VOA)$, but it is not known if they form a spanning set.

We proceed to review the pseudo-trace functions of \cite{Arike:2011ab}.

\medskip

For a $k$-algebra $A$, denote by
\be
	C(A) ~=~ \big\{ \,\varphi : A \to k \,\big|\, \varphi(ab) = \varphi(ba) \text{ for all }a,b \in A \,\big\} 
\ee
the space of {\em central forms} on $A$. 
	By definition, for $\varphi \in C(A)$, the pairing $(a,b) \to \varphi(ab)$ is symmetric. If it is also non-degenerate, $\varphi$ turns $A$ into a symmetric Frobenius algebra.
We recall the following simple lemma (see e.g.\ \cite[Lem.\,2.5]{Broue:2009}):

\begin{lemma}\label{lem:ZA-CA-iso}
Let $A$ be a symmetric Frobenius algebra over $k$. Let $\eps : A \to k$ induce the non-degenerate pairing on $A$. Then the map $Z(A) \to C(A)$, $z \mapsto \eps(z \cdot (-))$, is an isomorphism.
\end{lemma}

Let 
$\varphi \in C(A)$ be a central form. For each finitely-generated projective $A$-module $P$, we define a trace function $t^\varphi_P : \End_A(P) \to k$,
	also called a Hattori-Stallings trace \cite{Hattori:1965,Stallings:1965},
as follows.
Choose a finite-dimensional vector space $X$ such that there is a surjective $A$-module map $\pi : A \ot X \to P$. Since $P$ is projective, 
there is an $A$-module map $\iota : P \to A \ot X$ which is right-inverse to $\pi$, i.e.\ $\pi \circ \iota = \id_P$. 
Define $t^\varphi_{P}(f)$ as the image of $1 \in A$ under the map
\begin{align}
	& A = A \ot k 
	\xrightarrow{\id \ot \coev_X} A \ot X \ot X^*
	\xrightarrow{\pi \ot \id} P \ot X^*
	\xrightarrow{f \ot \id} P \ot X^*
\nonumber\\&
	\xrightarrow{\iota \ot \id} A \ot X \ot X^*
	\xrightarrow{\id \ot \widetilde\ev_X} A \ot k = A
	\xrightarrow{\varphi} k \ .
\label{eq:trace-map-def}
\end{align}
One proves that this definition is independent of the choice of $X$, $\pi$ and $\iota$ (see e.g.\ \cite{Arike:2010} and references therein).
It is not difficult to verify the following properties of $t^\varphi$ (all tensor products are over $k$):
\begin{enumerate}
\item For all finitely-generated projective $A$-modules $P$, finite-dimensional $k$-vector spaces $W$ and $f \in \End_A(P \ot W)$ we have that $t^\varphi_{P \ot W}(f)$ is equal to $t^\varphi_{P}$ applied to the partial trace
\be
	\big[\,
	P = P \ot k
	\xrightarrow{\id \ot \coev_W} P \ot W \ot W^*
	\xrightarrow{f \ot \id} P \ot W \ot W^*
	\xrightarrow{\id \ot \widetilde\ev_W} P \ot k = P
	\,\big] \ .
\ee
\item For all finitely-generated projective $A$-modules $P,Q$ and $f : P \to Q$, $g:  Q \to P$ we have
\be
	t^\varphi_Q(f \circ g) = t^\varphi_P(g \circ f) \ .
\ee
\end{enumerate}
For example, 
property 1 follows if, given $A \ot X \xrightarrow{\pi} P \xrightarrow{\iota} A \ot X$ one chooses $X' = X \ot W$, $\pi' = \pi \ot \id_W$, $\iota' = \iota \ot \id_W$. 

\medskip

Let $\VOA$ be a $C_2$-cofinite and non-negatively-graded VOA. The {\em zero mode} $o(v) \in \End(\VOA)$ of a homogeneous element $v \in \VOA_h$ ($h \in \Zb_{\ge 0}$) is defined as the coefficient
\be
	Y(v,x) ~=~ o(v) \, x^{-h} + \text{ (other powers of $x$) } \ .
\ee
We extend $o(v)$ linearly to all of $\VOA$.
Since $[L_0,o(v)]=0$, when acting on a $\VOA$-module, $o(v)$ will preserve generalised $L_0$-eigenspaces.

Let $M = \bigoplus_{h \in \Cb} M_h$ be a finitely-generated $\VOA$-module. The grading is by generalised $L_0$-eigenspaces, and $M_h$ is non-zero only for a countable subset of $\Cb$. Since $M$ is finitely-generated, each $M_h$ is finite-dimensional
(see \cite{Gaberdiel:2000qn} and \cite[Lem.\,2.4]{Miyamoto:2002ar}). 

Let $E := \End_\VOA(M)$ be the $k$-algebra of $\VOA$-intertwiners of $M$. Then $M$ is an $E$-module. 
Suppose further that $M$ is projective as an $E$-module.\footnote{
Instead of putting a condition on $M$,
in \cite{Arike:2011ab}, a subalgebra of $E$ such that $M$ is projective over that subalgebra is used instead (called `projective commutatant' there).
We omit passing to a subalgebra here because in Section \ref{sec:voa} we will need all of $\End_\VOA(M)$ for a projective generator $M$. In this case projectivity as an $\End_\VOA(M)$ is automatic (Proposition \ref{prop:G-EndG-is-proj}).}
Write $c$ for the central charge of $\VOA$ and fix $\varphi \in C(E)$. The {\em pseudo-trace function} $\xi^\varphi_M$ is defined as, for $v \in V$, $\tau \in \mathbb{H}$,
\be\label{def:pseudo-trace}
	\xi^\varphi_M(v,\tau)
	\,=\, t^\varphi_M\big(\,o(v) \,e^{2 \pi i \tau (L_0 - c/24)}\, \big)
	\,=\,
	\sum_{h \in \Cb}
	t^\varphi_{M_h}\big(\,o(v)\, e^{2 \pi i \tau (L_0 - c/24)}\,\big) \, 
	 \ ,
\ee
where 
$M_h$ is obviously a finitely-generated $E$-module and
 the last expression serves as a definition of the trace over the typically not-finitely-generated $E$-module $M$. The sum in the last term is countable (as $M$ is finitely generated as a $\VOA$-module). 
We have:

\begin{proposition}[{\cite{Miyamoto:2002ar}, \cite[Thm.\,4.3.4]{Arike:2011ab}}]
\label{prop:pt-torus1}
Let $\VOA$ be $C_2$-cofinite and non-negatively graded, and let $M$ be a finitely generated $\VOA$-module that is projective as an $\End_\VOA(M)$-module. Then the assignment $\varphi \mapsto \xi^\varphi_M$ defines a map $C(\End_\VOA(M)) \to C_1(\VOA)$.
\end{proposition}

No claim about injectivity or surjectivity (even as $M$ varies) is made in \cite{Arike:2011ab}.\footnote{However, surjectivity is shown for the (differently defined) pseudo-trace functions in \cite{Miyamoto:2002ar}.} This will necessitate Conjecture \ref{conj:C(End)-C1(V)-iso} below.

\section{Conjectural modular Verlinde formula}\label{sec:voa}

This section is the main part of the paper. We conjecture a relation between the categorical construction in Section \ref{sec:GrVer} and the modular properties of the pseudo-trace functions in Section \ref{sec:ptfun}. This allows one to compute the structure constants of the Grothendieck ring from the $S$-transformation of pseudo-trace functions.

\medskip

In this section, let $\VOA$ be a VOA which is
\begin{enumerate}
\item
$C_2$-cofinite,
\item non-negatively graded and satisfies $\VOA_0 = \Cb 1$ (this is sometimes called ``of CFT type'' or ``of positive energy''),
\item simple as a $\VOA$-module,
\item isomorphic as a $\VOA$-module to the contragredient module $\VOA'$ (this amounts to the existence of a non-degenerate invariant pairing).
\end{enumerate}

We define
\be 
	\Rep(\VOA)
\ee 
to be the category consisting of all $\VOA$-modules $M = \bigoplus_{h \in \Cb} M_h$, graded by generalised $L_0$-eigenspaces, with the following property:
For each real number $r$ the sum $\sum_{h \in \Cb, \text{Re}(h)<r} \dim M_h$ is finite. Such modules are called {\em quasi-finite-dimensional generalised $\VOA$-modules}, see \cite[Def.\,1.2]{Huang:2007mj}.

Collecting the results of Propositions 4.1, 4.3 and Theorems 3.24, 4.11 of 
\cite{Huang:2007mj}, which in turn depend on the logarithmic tensor product theory of \cite{Huang:2010}, we get:

\begin{theorem}[\cite{Huang:2007mj,Huang:2010}]
\label{thm:RepV-properties}
The category $\Rep(\VOA)$
\begin{enumerate}
\item
is $\Cb$-linear finite abelian,
\item
is braided monoidal with $\Cb$-bilinear tensor product functor,
\item
has a simple tensor unit.
\end{enumerate}
\end{theorem}

Since $\Rep(\VOA)$ is finite abelian, it has a projective generator $G$. We abbreviate
the centraliser of the $\VOA$-action on $G$ as
\be
E := \End_\VOA(G).
\ee
 Then $\Rep(\VOA)$
  is adjointly equivalent to the category of finite-dimensional right $E$-modules\footnote{
See e.g.\ \cite[Sect.\,1.8]{EGNO-book} for the equivalence $\Hom_\VOA(G,-)$. The inverse equivalence is the standard Hom-tensor adjunction.}
\be\label{eq:RepV-modE-equiv}
\xymatrix@C=80pt@W=5pt@M=5pt{
\Rep(\VOA) \ar@/^/[r]^{\funF:=\Hom_{\VOA}(G,-)}
&\text{mod-}E \ar@/^/[l]^{- \otimes_E G}
\ .
}
\ee
For the functor $- \otimes_E G$, $G$ is considered as the left-left $(\VOA,E)$-bimodule. For $M \in \text{mod-}E$, we indeed have that $M \otimes_E G$ is a quasi-finite-dimensional generalised $\VOA$-module. To see this, note that $M \otimes_E G$ can be written as a cokernel
of the difference $l-r$ in
\be\label{eq:tensor_E-cokernel}
	M \otimes_\Cb E \otimes_\Cb G \xrightarrow{\; l-r\; } M \ot_\Cb G \xrightarrow{\; \pi_\ot\; } M \otimes_E G \ ,
\ee
where $l$, $r$ denote the left and right $E$-actions.
 Since the tensor products over $\Cb$ just give finite direct sums of $G$, and since $\Rep\VOA$ contains cokernels, $M \otimes_E G \in \Rep(\VOA)$.

\medskip

For the definition of pseudo-trace functions below we will need:

\begin{proposition}\label{prop:G-EndG-is-proj}
$G$ is projective as a
	left $E$-module
in the category 
of (possibly infinite dimensional) left $E$-modules.
Equivalently, if $G = \bigoplus_{h \in \Cb} G_h$ is the decomposition of $G$ into generalised $L_0$-eigenspaces, each $G_h$ is projective in the category of finite-dimensional left $E$-modules.
\end{proposition}

The proof relies on the following lemma, due to \cite{Meir-priv}.

\begin{lemma}\label{lem:G-EndG-is-proj}
Let $\mathcal{A}$ be a $k$-linear finite abelian category and let $D : \mathcal{A} \to \Vect^\mathrm{fd}(k)$ be an exact 
	$k$-linear
functor to finite-dimensional $k$-vector spaces. Let $G$ be a projective generator of $\mathcal{A}$. Then $D(G)$ is projective as an $\End_\mathcal{A}(G)$-module.
\end{lemma}

\begin{proof}
Write $E := \End_\mathcal{A}(G)$. Since $\mathcal{A}$ is finite abelian, $E$ is a finite-dimensional $k$-algebra. Denote by mod-$E$ the category of finite-dimensional right $E$-modules. 
	For the same reason as in \eqref{eq:RepV-modE-equiv},
the functor
 ${\mathcal{A}}(G,-) : \mathcal{A} \to \text{mod-}E$ is a $k$-linear equivalence.
The composition 
$H := D \circ {\mathcal{A}}(G,-)^{-1}$ is a $k$-linear exact functor. 

The bimodule structure ${}_EE_E$ turns $H(E)$ into a left $E$-module.
Since $H$ is in particular right-exact, one obtains a natural isomorphism $H \cong (-) \otimes_E H(E)$. 
Exactness of $H$ then implies that the left $E$-module $H(E)$ is flat. Since $H(E)$ is finite-dimensional, it is also projective as a left $E$-module.

By construction, the $E$-actions on $H(E)=D(G)$ agree. Hence $D(G)$ is a projective $E$-module.
\end{proof}

\begin{proof}[Proof of Proposition \ref{prop:G-EndG-is-proj}]
For $h \in \Cb$, let $D_h : \Rep(\VOA) \to \Vect(\Cb)$ be the functor which takes a $\VOA$-module $M$ to its generalised $L_0$-eigenspace of eigenvalue $h$. 
Since morphisms  
	commute with the $L_0$-action,
this is indeed a functor which is furthermore exact. Since modules in $\Rep(\VOA)$ are quasi-finite dimensional, the image of $D_h$ lies in finite-dimensional $\Cb$-vector spaces. Finally, by Theorem \ref{thm:RepV-properties}, $\Rep(\VOA)$ is finite. 
Applying Lemma \ref{lem:G-EndG-is-proj} to $\mathcal{A} = \Rep(\VOA)$ and $D = D_h$ shows that $D_h(G)$ is projective as an
	 $E$-module. 
Hence, so is the
 $E$-module 
$G = \bigoplus_{h \in \mathbb{C}} D_h(G)$. 
\end{proof}

\begin{remark}\label{rem:morita}
Proposition~\ref{prop:G-EndG-is-proj} allows one to state a version of Morita equivalence between $\VOA$ and its centraliser on $G$, the algebra  $E$. Indeed, $G$ is projective  for both the algebras and it is a 
projective generator in the category of possibly  infinite-dimensional  
	right
$E$-modules. To show the last claim,
we decompose  $G=\bigoplus_{U\in\Irr(\VOA)} n_U P_U$, as a left $\VOA$-module. There is a natural projection to $\bigoplus_{U\in\Irr(\VOA)} n_U U$ and it commutes with the action of $E$ divided by its Jacobson radical. Therefore, $G$ covers any irreducible $E$-module. Since by Proposition~\ref{prop:G-EndG-is-proj} $G$ is projective, it is a projective generator for $E$-modules.
\end{remark}

For each $M \in \Rep(\VOA)$, we define a central form $\varphi_M \in C(E)$ via
\be\label{eq:varphiM-via-tr}
	\varphi_M : E \longrightarrow \Cb
	\quad , \quad f \,\mapsto\, \mathrm{tr}_{\mathcal{F}(M)}(f) \ .
\ee
Here, $\mathcal{F}$ is the equivalence from \eqref{eq:RepV-modE-equiv} and $\mathrm{tr}_{\mathcal{F}(M)}$ is the vector-space trace over the finite-dimensional left $E$-module $\mathcal{F}(M)$.

By Proposition \ref{prop:G-EndG-is-proj}, we can define pseudo-trace functions for $G$ and any central form $\varphi \in C(E)$ as in \eqref{def:pseudo-trace}. For the forms $\varphi_M$ just defined, one obtains:

\begin{proposition}\label{prop:tr_m-vs-pseudo_G}
For $M \in \Rep(\VOA)$, the pseudo-trace function for $\varphi_M$ satisfies \be
	\xi^{\varphi_M}_G(v,\tau)
	~=~ 
	\mathrm{tr}_{M}\big(\,o(v) \,e^{2 \pi i \tau (L_0 - c/24)}\, \big) 
	\qquad , ~~ v \in \VOA~,~\tau \in \mathbb{H} \ ,
\ee
where $\mathrm{tr}_{M}$ is the vector space trace over $M$.
\end{proposition}

\begin{proof}
The pseudo-trace function $\xi^\varphi_G$ is defined in terms of $E$-module maps $\pi$, $\iota$. We choose $\pi$ to be a map $E \ot G \to G$, and to be given by acting with $E$ on $G$, $\pi(f \ot g) = f(g)$. For $\iota$ we choose an arbitrary right-inverse $E$-module map, $\pi \circ \iota = \id_G$. 
We abbreviate $M' := \mathcal{F}(M)$ and write $\rho' : M' \ot E \to M'$ for the right action of $E$ on $M'$.
Let
\be
	p ~=~ \big[ M' \ot G \xrightarrow{\id \ot \iota} M' \ot E \ot G 
	\xrightarrow{ \rho' \ot \id } M' \ot G \, \big] \ .
\ee
Let $l,r,\pi_\ot$ be as in \eqref{eq:tensor_E-cokernel}.
Using that $\iota$ is an $E$-module map, and that  $\pi \circ \iota = \id$, one verifies that
\be
	p \circ (l-r) = 0
	~~,\quad
	\pi_\otimes \circ p = \pi_\otimes \ .
\ee
The first property implies that there is a map
 $\bar p : M' \ot_E G \to M' \ot G$
  such that $p = \bar p \circ \pi_\ot$. From the second condition we learn that $\pi_\otimes \circ \bar p \circ \pi_\ot = \pi_\otimes$, and thus, since $\pi_\ot$ is surjective, that
\be\label{eq:pi-barp=id}
	\pi_\otimes \circ \bar p = \id_{M' \ot_E G} \ .
\ee
Abbreviate $\mathcal{O} = o(v) \,e^{2 \pi i \tau (L_0 - c/24)}$. Since 
by \eqref{eq:RepV-modE-equiv}, $M \cong M' \ot_E G$, we have
\begin{align}
\mathrm{tr}_{M}\big(\,\mathcal{O}\, \big)
&=
\mathrm{tr}_{M' \ot_E G}\big(\,\mathcal{O}\, \big)
\overset{\eqref{eq:pi-barp=id}}=
\mathrm{tr}_{M' \ot_E G}\big(\,\pi_\otimes \,\bar p \, \mathcal{O}\, \big)
\overset{(*)}=
\mathrm{tr}_{M' \ot G}\big(\,\bar p\, \pi_\otimes\, \mathcal{O}\, \big)
\nonumber\\
&=
\mathrm{tr}_{M' \ot G}\big(\,p \, \mathcal{O}\, \big)
\overset{(**)}=
t^{\varphi_M}_G\big(\,\mathcal{O}\, \big) \ .
\end{align}
In step (*),  we used the cyclicity property of the trace and the fact that $\pi_\otimes$ is a $\VOA$-module map, while
in step (**) the partial trace over $M'$ was taken.
\end{proof}

In particular, for each irreducible $\VOA$-module $U$, we have
\be\label{varphiU}
	\xi^{\varphi_U}_G ~=~ \xi^{\id^*}_U \ , 
\ee
where $\id^* : \End_\VOA(U) \to \Cb$ is the form which takes $\id_U$ to $1$ (since $U$ is simple, $\End_\VOA(U) = \Cb \id_U$). 

\begin{remark}
Although we do not need it in this paper, let us point out that Proposition~\ref{prop:tr_m-vs-pseudo_G} generalises to all pseudo-trace functions as follows. Let $M \in \Rep(\VOA)$ and let $P \subset \End_\VOA(M)$ be a subalgebra such that $M$ is projective as a $P$-module ($P$ is called `projective commutant' in \cite{Arike:2011ab}). Then for each $\psi \in C(P)$ there exists a $\varphi \in C(E)$ such that
\be
	\xi^{\varphi}_G(v,\tau)
	~=~ 
	t^\psi_{M}\big(\,o(v) \,e^{2 \pi i \tau (L_0 - c/24)}\, \big) 
	\qquad , ~~ v \in \VOA~,~\tau \in \mathbb{H} \ .
\ee
To see this, one first shows that $M$ is projective as a $P$-module if and 
only if $\mathcal{F}(M)$ is projective as a $P$-module.\footnote{
Let $M' = \mathcal{F}(M)$,
recall the definition in~\eqref{eq:RepV-modE-equiv}.
Suppose $M'$ is projective as a left $P$-module. Pick a $P$-module map $\iota : M' \to P \ot M'$ right-inverse to the $P$-action on $M'$. Combine $\iota$ with $\bar p$ and $\pi_\ot$ from the proof of Proposition \ref{prop:tr_m-vs-pseudo_G} to construct $P$-module maps 
$\tilde\pi : P \ot M' \ot G \to M' \ot_E G$ and $\tilde\iota : M' \ot_E G \to P \ot M' \ot G$ with $\tilde\pi \circ \tilde\iota = \id$.
This establishes
$M' \ot_E G\cong M$  as a direct summand of the free $P$-module $P \ot M' \ot G$. Conversely, if $M \cong M' \ot_E G$ is projective as a left $P$-module, use that by  Remark \ref{rem:morita}, $G$ is a generator for left $E$-modules. Then $G^{\oplus m} \cong E \oplus X$ as $E$-modules for some $m$ and $X$. Thus
  $M^{\oplus m} \cong M' \ot_E G^{\oplus m} \,\cong\, M' \,\oplus\, M' \ot_E X$ 
 as $P$-modules, showing that
  $M'$ is a direct summand of a projective $P$-module.
 }
Then one defines the pseudo-trace in terms of maps $\pi' : P \ot \mathcal{F}(M) \rightarrow \mathcal{F}(M)$ and $\iota' : \mathcal{F}(M) \to P \ot \mathcal{F}(M)$, resulting in an explicit formula for $\varphi$, namely
$\varphi(f) = \mathrm{tr}_{\mathcal{F}(M)}\big( (\psi \ot \id) 
	\circ \iota' \circ f \big)$,
where $f \in E$. 

The above argument shows that all pseudo-trace functions can be obtained from those for a given choice of a projective generator. 
\end{remark}

It is not known in general whether $\Rep(\VOA)$ is rigid (it is in the semisimple case \cite{Huang2005}) or pivotal, so we add this as a conjecture (cf.\ \cite[Conj.\,4.2\,\&\,Cor.\,4.3]{Huang:2009xq}). 
Furthermore, again from analogy to the semisimple case \cite{Huang2005}, and from the hoped-for relation with 3d\,TFT (see Remark \ref{rem:bizarre-conjecture} below), we expect $\Rep(\VOA)$ to be factorisable. Altogether:

\begin{conjecture}\label{conj:fact}
$\Rep(\VOA)$ is rigid, pivotal and factorisable.
\end{conjecture}

Already from Theorem \ref{thm:RepV-properties} we know that $\Rep(\VOA)$ has a projective generator $G$. 
By Proposition \ref{prop:G-EndG-is-proj} we can define pseudo-trace functions for $G$.
The next conjecture we will need is:

\begin{conjecture}\label{conj:C(End)-C1(V)-iso}
Let $G$ be a projective generator of $\Rep(\VOA)$.
Then the map $C(E) \to C_1(\VOA)$ from Proposition \ref{prop:pt-torus1} is an isomorphism of $\Cb$-vector spaces.
\end{conjecture}

In the symplectic fermion example treated in Section~\ref{sec:sf}  the conjecture is known to be true for $N=1$ \cite{Flohr:2005cm,Arike:2011ab}.
For the triplet $W_{1,p}$ VOA, it is shown in~\cite{AM09} that the space spanned by all the vacuum pseudo-traces (\textit{i.e.}, for $v=\one$) is $(3p-1)$-dimensional
	which
agrees with the dimension of the centre of the centraliser $E$ (for any choice of the projective generator $G$) and 
	that of
the space $C(E)$.

The motivation
for Conjecture~\ref{conj:C(End)-C1(V)-iso}
comes from the relation to 3d\,TFT, as explained in the next remark.

\begin{remark}\label{rem:bizarre-conjecture}
In the semisimple case, there is a beautiful relation between chiral 2d conformal field theory (CFT) and 3d topological field theory, as first studied in \cite{Witten:1988hf,Moore:1989vd}. One mathematical precise formulation of this relation is that for a VOA which satisfies the properties in the beginning of this section, and which is in addition rational, the category $\Rep(\VOA)$ is a modular tensor category \cite{Huang2005}. Each modular tensor category in turn defines a 3d TFT \cite{retu,tur}. For a given surfaces, possibly with marked points, the spaces of conformal blocks of the VOA are expected to be isomorphic to the state space of the corresponding 3d TFT. This is known to be true in genus 0 and 1.

In the non-semisimple case much less is known. A variant of a 3d TFT which is constructed from a factorisable finite ribbon category $\Cc$ is given in \cite{Lyubashenko:1994tm,Kerler:2001}. The application of this theory to
	non-semisimple (aka.\ logarithmic)
 conformal field theory has been developed starting with \cite{Fuchs:2010mw}, see \cite{Fuchs:2013lda,Fuchs:2016wjr,Fuchs:2016whb} for recent results and 
further references.
In the 3d TFT for $\Cc$, the state space of the torus is $\Cc(\one,L)$. If this 3d TFT is indeed related to the chiral CFT for $\VOA$, the following must hold: 
\begin{enumerate}
\item $\Rep(\VOA)$ must be rigid, pivotal (actually: ribbon) and factorisable,
\item $\Hom_\VOA(\one,L)$ must be isomorphic to $C_1(\VOA)$,
\item for an appropriate choice of the isomorphism in 2., the projective $SL(2,\Zb)$-action on $\CF(\Rep(\VOA)) = \Hom_\VOA(\one,L)$ agrees (projectively) with the action on $C_1(\VOA)$.
\end{enumerate}
Condition 1 is covered by Conjecture \ref{conj:fact}. 
Condition 2 follows from Conjecture \ref{conj:C(End)-C1(V)-iso} and Lemmas \ref{lem:EndId-CE}, \ref{lem:Omalg}, together with Lemma \ref{lem:ZA-CA-iso} and Conjecture \ref{conj:S-agrees} below. 
Of the last condition we just need compatibility of the $S$-transformation, see again Conjecture \ref{conj:S-agrees}.
\end{remark}

Let us assume Conjecture \ref{conj:C(End)-C1(V)-iso} holds.
Let $S = \big( \begin{smallmatrix} 0 & -1 \\ 1 & 0 \end{smallmatrix} \big) \in SL(2,\Zb)$ 
be the generator corresponding to the modular $S$-transformation $\tau \mapsto -1/\tau$. By Theorem \ref{thm:C1-mod-inv}, $\xi \mapsto \xi|_S$ is an automorphism of $C_1(\VOA)$. Via the isomorphism in Conjecture \ref{conj:C(End)-C1(V)-iso}, there exists a unique linear automorphism
\be\label{eq:SVOA-on-CE}
	S_{\VOA} ~:~ C(E) \longrightarrow C(E) \ ,
\ee
such that for all $\varphi \in C(E)$ we have $\xi^\varphi_G|_S = \xi^{S_\VOA(\varphi)}_G$.

Each central form $\delta \in C(E)$ whose associated pairing on $E$ is non-degenerate provides an isomorphism $\hat\delta : \End(\Id_{\Rep(\VOA)}) \to C(E)$ via Lemma \ref{lem:ZA-CA-iso}. 
This allows one to compare $S_\VOA$ and $\varphi_U$ 
defined in~\eqref{eq:varphiM-via-tr}, or equivalently in~\eqref{varphiU},
to the categorical modular transformation $S_{\Rep(\VOA)}$  from Section \ref{sec:GrVer} and $\tildechi_U$ defined in~\eqref{eq:tildechi-def}.
Recall from Section \ref{sec:non-deg} that the integral $\Lambda_L$ is determined only up to a sign. Thus also the maps $\Rad$, $S_\Cc$ and the elements $\tildechi_U$ are determined only up to an overall sign.
The idea is now to find a choice of $\delta$ 
and of the sign of $\Lambda_L$
such that $\hat\delta(\tildechi_U) = \varphi_U$ for all $U \in \Irr(\Rep(\VOA))$ and such that
\be\label{eq:S-compatible}
	\xymatrix{
	\End(\Id_{\Rep(\VOA)}) \ar[rr]^{S_{\Rep(\VOA)}} \ar[d]^{\hat\delta} &&
	\End(\Id_{\Rep(\VOA)}) \ar[d]^{\hat\delta}
	\\
	C(E) \ar[rr]^{S_{\VOA}} &&
	C(E)}
\ee
commutes. Recall from point 2 of Remark \ref{rem:ver}
 that  $S_\Cc(\tildechi_\one) = \id$ and observe that by definition, $\hat\delta(\id) = \delta$. Thus, if the above idea is to work, we must have that $S_\VOA(\varphi_\VOA) = \delta$.
This leads us to the main conjecture:

\begin{conjecture}\label{conj:S-agrees}
Let $G \in \Rep(\VOA)$ be a projective generator. Then
\begin{enumerate}
\item
$\delta := S_\VOA(\varphi_\VOA)$ induces a non-degenerate pairing on 
$E = \End_\VOA(G)$, 
i.e.\ the latter becomes a symmetric Frobenius algebra.
\item
There is a choice for the sign of $\Lambda_L$ such that the isomorphism $\hat\delta : \End(\Id_{\Rep(\VOA)}) \to C(E)$ 
makes the diagram \eqref{eq:S-compatible} commute and
maps $\tildechi_U$ to $\varphi_U$ for each $U \in \Irr(\Rep(\VOA))$.
\end{enumerate}
\end{conjecture}

Assuming the above conjectures, let us give the procedure to compute the structure constants of the Grothendieck ring $\Gr(\Rep(\VOA))$.

\begin{enumerate}
\item Compute the irreducible $\VOA$-modules $U \in \Irr(\Rep(\VOA))$ and their projective covers $P_U$. Fix a projective generator $G$ of $\Rep(\VOA)$.
\item Compute $\varphi_U \in C(E)$ for all $U \in \Irr(\Rep(\VOA))$
	via~\eqref{eq:varphiM-via-tr}. 
Compute the modular transformation $S_\VOA$
on $\varphi_U$. Let $\delta = S_\VOA(\varphi_\VOA)$.

\item Compute $\End(\Id_{\Rep(\VOA)}) \cong Z(\End_\VOA(G))$ and transport the composition in $\End(\Id_{\Rep(\VOA)})$ to $C(\End_\VOA(G))$ via $\hat\delta$: 
for all $\alpha,\beta \in C(\End_\VOA(G))$ let 
\be
	\alpha \muldiag \beta := \hat\delta\big( \, \hat\delta^{-1}(\alpha) \circ
\hat\delta^{-1}(\beta) \, \big) \ .
\ee
\item For $A,B \in \Irr(\Rep(\VOA))$ compute the unique constants $N_{AB}^{~C}$ in
\be\label{eq:compute-NABC}
	S_\VOA^{-1}\big( \, S_\VOA(\varphi_A) \muldiag S_\VOA(\varphi_B) \, \big)
	~= \hspace{-1em}
	\sum_{C \in \Irr(\Rep(\VOA))} \hspace{-1em} N_{AB}^{~C} ~ \varphi_C \ .
\ee
\end{enumerate}

\begin{theorem}\label{thm:main}
Let $\VOA$ be a VOA as stated in the beginning of this section and assume Conjectures \ref{conj:fact}, \ref{conj:C(End)-C1(V)-iso}, \ref{conj:S-agrees}. 
Then steps 1--4 compute the structure constants of $\Gr(\Rep(\VOA))$.
\end{theorem}

\begin{proof}
By Conjecture \ref{conj:S-agrees}, we have $\hat\delta(\tildechi_U) = \varphi_U$, $U \in \Irr(\Rep(\VOA))$. Since, by the same conjecture, diagram \eqref{eq:S-compatible} commutes, it follows that $S_{\Rep(\VOA)}(\tildechi_U) = S_\VOA(\varphi_U)$. 

By Theorem \ref{thm:RepV-properties} and Conjecture \ref{conj:fact}, $\Rep(\VOA)$ satisfies the assumptions of Theorem \ref{thm:ver}. 
Existence and uniqueness of the $N_{AB}^{~C}$ in \eqref{eq:compute-NABC} and the statement that they give the structure constants of the Grothendieck ring now follow from \eqref{eq:ver}.
\end{proof}

\begin{remark}{~}\\[-1.5em]
\label{rem:comments-ver}
\begin{enumerate}
\setlength{\leftskip}{-1em}
\item
We have listed steps 1--4 in this form to stress that knowing the irreducible and projective $\VOA$-modules, the pseudo-trace functions for a projective generator, as well as their behaviour under $\tau \mapsto -1/\tau$ are enough to determine $\Gr(\Rep(\VOA))$. 
	Apart from finding the projective covers in step 1, 
the main difficulty is to compute the $S_\VOA(\varphi_U)$, i.e.\ the central forms corresponding to the $S$-transformation of the characters of the irreducible $\VOA$-modules $U$.
\item
Recall from point 1 of Remark \ref{rem:ver} 
that -- under the assumptions in Theorem \ref{thm:main} -- to evaluate \eqref{eq:compute-NABC} it is enough to know $S_\VOA$ on the linear span of the 
$\varphi_U$, $U \in \Irr(\Rep(\VOA))$. For example, for $N$ pairs of symplectic fermions as treated in Section \ref{sec:sf}, the span of the $\varphi_U$ is four-dimensional, while the dimension of $C(\End_\VOA(G))$ is 
at least $2^{2N-1}+3$ 
(and 	is
conjecturally equal to this number, cf.\ Section \ref{sec:sf}).

\item
In his fundamental paper, Verlinde \cite{Verlinde:1988sn} conjectured the relation between fusion rules and modular properties of characters which now carries his name. See e.g.\ \cite{Fuchs:2006nx} for more details and references regarding proofs of the Verlinde formula in different settings. The first proof valid for all VOAs satisfying the conditions in the beginning of this section, together with rationality (so that $\Rep(\VOA)$ is semisimple) was given in \cite[Thm.\,5.5]{Huang:2004bn}.
\item
Starting with \cite{Flohr:2001zs} and with \cite{Fuchs:2003yu}, a number of works have considered finitely non-semisimple generalisations of the Verlinde formula for the structure constants of the Grothendieck ring, see \cite{Fuchs:2006nx,Flohr:2007jy,Gaberdiel:2007jv,Gainutdinov:2007tc,Pearce:2009pg}. All of these works are concerned with the $W_{1,p}$-triplet models,
for  integer $p\geq2$. 
For these models, a relation similar to~\eqref{eq:S-compatible} was stated in~\cite{[FGST]} -- an equivalence between $SL(2,\Zb)$-actions on the space of 
	vacuum pseudo-traces
for the $W_{1,p}$ VOA 
and on the space of skew-symmetric forms on the centraliser of the
$W_{1,p}$-action on a projective generator.

To our knowledge, the procedure in Theorem \ref{thm:main}, and the conjectures it relies on, give the first precise statement on how to compute fusion rules from modular data for a general class of VOAs with non-semisimple representation theory.

In a promising different approach -- reviewed in \cite{Ridout:2014yfa} -- fusion rules are computed by first passing to a non-$C_2$-cofinite sub-VOA and using a continuum version of the semisimple Verlinde formula.
\end{enumerate}
\end{remark}

\begin{remark}\label{rem:ssi-steps14}
Let us illustrate steps 1--4 in the case that $\VOA$ is in additional rational, where they produce the standard modular Verlinde formula. Along the way, we point out why the conjectures entering Theorem \ref{thm:main} hold in this case.
\begin{enumerate}
\setlength{\leftskip}{-1em}
\item
 By \cite{Huang2005}, $\Cc := \Rep(\VOA)$ is a modular tensor category. In particular, Conjecture \ref{conj:fact} holds.
We have $P_U = U$ for all $U \in \Irr(\Cc)$ and we may choose $G = \bigoplus_{U \in \Irr(\Cc)} U$, so that $\End(G) = \bigoplus_{U \in \Irr(\Cc)} \Cb\,\id_U$.
Note that the $\id_U$ are mutually orthogonal central idempotents.
Conjecture \ref{conj:C(End)-C1(V)-iso} follows from \cite{Zhu1996}.
\item
The irreducible characters are obtained for $\varphi_U(\id_X) = \delta_{U,X}$ ($U,X \in \Irr(\Cc)$). Accordingly, $S_\VOA$ is the matrix describing the modular $S$-transformation of characters
\eqref{eq:ssi-irred-char-S},
$S_\VOA(\varphi_U) = \sum_{X \in \Irr(\Cc)} \mathbb{S}_{U,X} \, \varphi_X$. From this we see that
$\delta = \sum_{X \in \Irr(\Cc)} \mathbb{S}_{\one,X} \, \id_X^*$, which induces a non-degenerate pairing on $\End(G)$ since the $\mathbb{S}_{\one,X}$ are non-zero.

By the proof of \cite[Thm.\,4.5]{Huang2005},
the matrices $s_{A,B}$ and $\mathbb{S}_{A,B}$ are related by\footnote{
    The extra dual appears because the $S_a^b$ defined in \cite[Sect.\,2]{Huang2005} amount to our $\mathbb{S}_{W^a,W^b}$, but the composition of maps in \cite[Eqn.\,(4.1)]{Huang2005} is our $s_{W^a,(W^b)^*}$. In the Verlinde formula, these duals can be absorbed in the sum over $X$.
}
 $s_{A,B} = \mathbb{S}_{A,B^*} / \mathbb{S}_{\one,\one}$. 
From \eqref{eq:ssi-tildechi} we know that $\tildechi_U = (\mathrm{Dim}\Cc)^{\frac12} / \dim(U) \, \id_U \in \End(G)$. 
 We choose the sign of $\Lambda_L$ in \eqref{eq:ssi-LamL} by fixing $(\mathrm{Dim}\Cc)^{\frac12} = (\mathbb{S}_{\one,\one})^{-1}$. With this choice one immediately verifies that $\hat\delta(\tildechi_U) = \id_U^*$ and that \eqref{eq:S-compatible} commutes, i.e.\ that Conjecture \ref{conj:S-agrees} holds.

\item
The induced product is 
	$\varphi_A \muldiag \varphi_B = \delta_{A,B}
(\mathbb{S}_{\one,A})^{-1} \varphi_A$.
\item
Substituting into \eqref{eq:compute-NABC} produces
\be
    N_{AB}^{~C} = \sum_{X \in \Irr(\Cc)} \frac{\mathbb{S}_{A,X} \mathbb{S}_{B,X} (\mathbb{S}^{-1})_{X,C}}{\mathbb{S}_{\one,X}} \ .
\ee
Finally, one may use that $\mathbb{S}$ is symmetric and that $(\mathbb{S}^{-1})_{X,C} = \mathbb{S}_{C^*,X}$.
\end{enumerate}
\end{remark}

\section{Example: Symplectic fermions}\label{sec:sf}

In this section we apply steps 1--4 from Theorem \ref{thm:main} to the VOA $\VOA_\mathrm{ev}$ given by even part of symplectic fermions. 
The CFT of symplectic fermions \cite{Kausch:1995py} forms  an important example of so-called logarithmic CFTs which have a property that $L_0$ cannot be diagonalised (and which therefore involve non-semi\-simple representations). This type CFTs were first studied in \cite{Rozansky:1992rx,Gurarie:1993xq}.
Following indications in \cite{Gurarie:1993xq,Kausch:1995py}, in \cite{Gaberdiel:1996kx,Gaberdiel:1996np} the important observation was made that the tensor product of two simple modules (on which $L_0$ is necessarily diagonalisable) may result in a module on which $L_0$ cannot be diagonalised, but instead has finite-rank Jordan cells.
As we will review below, in this situation one needs a generalised notion of characters -- the pseudo-trace functions -- to describe the modular-group transformation properties of the CFT on a torus.

From a VOA point of view, symplectic fermions form a vertex operator super-algebra~$\VOA$. It depends on a parameter 
$$
 N \in \Zb_{>0} \ ,
$$
which we fix from now on. To describe a specific model, one can speak of ``$N$ pairs of symplectic fermions''. 
The even subspace $\VOA_\mathrm{ev} \subset \VOA$ is a VOA of central charge $c=-2N$.

\begin{theorem}[{\cite[Thm.\,3.12\,\&\,4.2]{Abe:2005}}]
$\VOA_\mathrm{ev}$ satisfies conditions 1--4 in the beginning of Section \ref{sec:voa}. 
\end{theorem}

Let us describe $\VOA_\mathrm{ev}$ and some of its modules in more detail. Let $\h$ be a symplectic $\Cb$-vector space of dimension $2N$ with symplectic form $(-,-)$. Define two affine Lie super-algebras -- $\widehat\h$ and $\widehat\h_\mathrm{tw}$ -- with underlying super-vector spaces
\be
	\widehat\h = \h \ot \Cb[t^{\pm1}]  \, \oplus \, \Cb \hat k
	\quad , \quad
	\widehat\h_\mathrm{tw} = \h \ot t^{\frac12}\Cb[t^{\pm1}] \, \oplus \, \Cb \hat k \ ,
\ee
where $t^{\pm1}$ and $\hat k$ are parity-even, and $\h$ is parity-odd. For $a \in \h$ and $m \in \Zb$ (resp.\, $m \in \Zb + \frac12$), abbreviate $a_m := a \ot t^m$. The Lie super-bracket is given by taking $\hat k$ central and setting, for $a,b \in \h$ and $m,n \in \Zb$ (resp.\ $m,n \in \Zb + \frac12$),
\be
	[a_m,b_n] \,=\, (a,b) \,m\, \delta_{m+n,0} \,\hat k \ .
\ee
Note that this is an anti-commutator as $a_m$, $b_n$ are parity odd.
We write
\be
	\Rep^\mathrm{fin}_{\flat,1}(\widehat\h)
	\qquad , \qquad
	\Rep^\mathrm{fin}_{\flat,1}(\widehat\h_\mathrm{tw})
\ee
for the categories of $\h_{(\mathrm{tw})}$-modules $M$ in $\sVect_\Cb$,
\begin{itemize}
\item (`$\flat$') which are bounded below in the sense that for each $x \in M$ there is an $L>0$ such that acting with any element of $U(\widehat\h_{(\mathrm{tw})})$ of degree $>L$ on $x$ gives zero,
\item (`1') on which $\hat k$ acts as $1$,
\item (`fin') which have a finite-dimensional space of ground states 
\be\label{eq:gnd-def}
M_\mathrm{gnd} := \{ x \in M \,|\, a_m x = 0 \text{ for all $m>0$ and $a \in \h$ }\} \ .
\ee
\end{itemize}

Let $\Lambda(\h)$ be the exterior algebra for $\h$. Use the even/odd $\Zb$-degree to turn $\Lambda(\h)$ into a super-algebra. As a super-algebra, $\Lambda(\h)$ is commutative. Let
$$
	\Rep^\mathrm{fd}(\Lambda(\h))
$$
be the category finite-dimensional $\Lambda(\h)$-modules in $\sVect_\Cb$. We note that, since the modules are taken in $\sVect_\Cb$, up to isomorphism $\Lambda(\h)$ has {\em two} simple modules: $\Cb^{1|0}$ and $\Cb^{0|1}$.

Define the untwisted and twisted induction $\widehat{(-)}$ as, for $A \in \Rep^\mathrm{fd}(\Lambda(\h))$ and $X \in \sVect_\Cb^\mathrm{fd}$,
\be\label{eq:induction-def}
\widehat A = 
U(\widehat\h) \otimes_{U(\widehat\h_{\ge 0} \oplus \Cb \hat k)} A
\quad , \quad
\widehat X = 
U(\widehat\h_\mathrm{tw}) \otimes_{U(\widehat\h_{\mathrm{tw},>0} \oplus \Cb \hat k)} X
\ . 
\ee
In both cases, $a_m$ with $m>0$ is taken to act as zero on the right hand factor in the tensor product. We have

\begin{proposition}[{\cite[Thm.\,2.4\,\&\,2.8]{Runkel:2012cf}}]
\label{prop:ind-equiv}
$\widehat{(-)}$ 
from \eqref{eq:induction-def}
and $(-)_\mathrm{gnd}$ 
from \eqref{eq:gnd-def}
are mutually inverse equivalences 
\be
\Rep^\mathrm{fd}(\Lambda(\h)) ~\cong~ \Rep^\mathrm{fin}_{\flat,1}(\widehat\h)
\qquad , \qquad
\sVect_\Cb^\mathrm{fd} ~\cong~ \Rep^\mathrm{fin}_{\flat,1}(\widehat\h_\mathrm{tw})
\ee
of $\Cb$-linear categories.
\end{proposition}

Let us abbreviate
\be
	\SF(\h) ~=~\SF_0 \oplus \SF_1
	\qquad \text{ with } \quad
	\SF_0 = \Rep^\mathrm{fd}(\Lambda(\h))
	~~,~~\SF_1 = \sVect_\Cb^\mathrm{fd} \ .
\ee
We denote the simple objects of $\SF(\h)$ as
\be\label{eq:SF-simples}
	\one = \Cb^{1|0} \in \SF_0
	~~,~~
	\Pi\one = \Cb^{0|1} \in \SF_0
	~~,~~
	T = \Cb^{1|0} \in \SF_1
	~~,~~
	\Pi T = \Cb^{0|1} \in \SF_1
	\ .
\ee
Here $\one$ is the tensor unit of a monoidal structure on $\SF(\h)$ (see Remark \ref{rem:SF-example} below), $\Pi$ is the parity-exchange functor, and $T$ stands for ``twisted''. Their projective covers are
\be\label{eq:SF-proj-covers}
	P_\one = \Lambda(\h)
	~~,~~
	P_{\Pi\one} = \Pi\Lambda(\h)
	~~,~~
	P_T = T
	~~,~~
	P_{\Pi T} = \Pi T \ .
\ee

Let $(-)_\mathrm{ev}$ be the functor from super-vector spaces to vector spaces consisting of taking the even subspace. 
	The even part of the symplectic fermion vertex operator super-algebra is given by $\VOA_\mathrm{ev} = (\widehat\one)_\mathrm{ev}$.
As a consequence of \cite{Abe:2005} we get:

\begin{proposition}\label{prop:hath-RepVev}
$(-)_\mathrm{ev} : \Rep^\mathrm{fin}_{\flat,1}(\widehat\h) \oplus \Rep^\mathrm{fin}_{\flat,1}(\widehat\h_\mathrm{tw}) \to \Rep(\VOA_\mathrm{ev})$ is a faithful $\Cb$-linear 
	exact
functor.
\end{proposition}

\begin{proof}
Each $\widehat\h$ or $\widehat\h_\mathrm{tw}$ module $M$ defines uniquely an action of all $v_n$, $v \in \VOA_\mathrm{ev}$, $n \in \Zb$ on $M_\mathrm{ev}$. It remains to show that $M_\mathrm{ev}$ is indeed a $\VOA_\mathrm{ev}$-module. 
In \cite[Sect.\,4.1\,\&\,5.1]{Abe:2005}, $\widehat{\Lambda(\h)}$ and $\widehat T$ are shown to be a $\VOA_\mathrm{ev}$-modules. As $\VOA_\mathrm{ev}$-modules we have $\widehat{\Lambda(\h)} = (\widehat{P_\one})_\mathrm{ev} \oplus (\widehat{P_{\Pi \one}})_\mathrm{ev}$ and $\widehat T = (\widehat{P_T})_\mathrm{ev} \oplus (\widehat{P_{\Pi T}})_\mathrm{ev}$. Since $\widehat{(-)}$ is an equivalence by 
Proposition \ref{prop:ind-equiv}, this shows that the images of all indecomposable projectives under $(-)_\mathrm{ev}$ give well-defined $\VOA_\mathrm{ev}$-modules. Since $(-)_\mathrm{ev}$ preserves quotients, $M_\mathrm{ev}$ is a $\VOA_\mathrm{ev}$-module for every $M \in \Rep^\mathrm{fin}_{\flat,1}(\widehat\h) \oplus \Rep^\mathrm{fin}_{\flat,1}(\widehat\h_\mathrm{tw})$.
\end{proof}

Combining Propositions~\ref{prop:ind-equiv} and~\ref{prop:hath-RepVev} we thus obtain:

\begin{corollary}\label{cor:faith-fun_SF-RepVev}
The $\Cb$-linear functor
\be\label{eq:faith-fun_SF-RepVev}
	(\widehat{-})_\mathrm{ev} : \SF(\h) \xrightarrow{~\sim~} \Rep( \VOA_\mathrm{ev})
\ee
is faithful.
\end{corollary}

To match the irreducibles in \eqref{eq:SF-simples} to those obtained in \cite{Abe:2005}, we give their lowest $L_0$-eigenvalue $h$ (conformal weight):
\begin{center}
\begin{tabular}{c|cccc}
repn. & $(\widehat\one)_\mathrm{ev}$ &  $(\widehat{\Pi\one})_\mathrm{ev}$ &  
$(\widehat T)_\mathrm{ev}$ &  $(\widehat{\Pi T})_\mathrm{ev}$ \\[.2em]
\hline
 \\[-.8em]
$h$ & 0 & 1 & $\tfrac{-N}8$ & $\tfrac{-N+4}8$
\end{tabular}
\end{center}

\medskip

We will now go through steps 1--4 in Section \ref{sec:voa}. In doing so, we first give the statements we can make that do not rely on any conjectures and then state the conjectures as we need them in order to finally apply Theorem~\ref{thm:main}. For example, Conjecture~\ref{conj:SF-RepV} below
states that the functor \eqref{eq:faith-fun_SF-RepVev} is an equivalence. (It is even conjectured to be a ribbon-equivalence, see Remark~\ref{rem:SF-example}.)

\subsubsection*{Step 1}

Consider the 
projective generator $G := P_\one \,\oplus\, P_{\Pi\one} \,\oplus\, T \,\oplus\, \Pi T \in \SF(\h)$ and the corresponding
$\VOA_\mathrm{ev}$-module
$\Gc := (\widehat{G})_\mathrm{ev}$. The conjectures below will imply that $\Gc$ is a projective generator of $\Rep( \VOA_\mathrm{ev})$, but we will not need this for now.
For the computation of pseudo-trace functions it will be convenient (as in \cite{Abe:2005,Arike:2011ab}) to rewrite $\Gc$ as
\be\label{eq:G-repn-for-trace-calc}
	\Gc ~=~ \widehat{\Lambda(\h)} \,\oplus\, \widehat{T} \ ,
\ee
	considered as a $\VOA_\mathrm{ev}$-module, rather than a $\VOA$-module.
Applying the faithful functor \eqref{eq:faith-fun_SF-RepVev} gives an injective algebra homomorphism
\be
	(\widehat{-})_\mathrm{ev} : 
	\End_{\SF}(G) \hookrightarrow \End_{\VOA_\mathrm{ev}}(\Gc) \ ,
\ee
whose image we denote by
 $\mathcal{E} \subset \End_{\VOA_\mathrm{ev}}(\Gc)$. It is shown in~\cite[Lem.\,6.2]{FGR2} that the subalgebra $\mathcal{E}$ can be described as 
\be\label{eq:E-EndVG-abbrev}
	\Ec = \Ec_0 \oplus \Ec_1
	\qquad \text{where}
	\quad
	\Ec_0 = \Lambda(\h)\rtimes \Cb\Zb_2
	~~,~~
	\Ec_1 = \Cb e_T \rtimes \Cb\Zb_2 \ .
\ee
Here, 
$e_T$ denotes the
idempotent for 	$\widehat T = (\widehat T)_\mathrm{ev} \oplus (\widehat{\Pi T})_\mathrm{ev}$, 
and $a \in \h$ 
acts on $\widehat{\Lambda(\h)}$ by the zero-mode $a_0$. In both summands, the generator $\kappa$ of $\Zb_2$ acts
by parity involution $\omega_{\Gc}$ on $\Gc$.
The parity involution $\omega_X$ on a super-vector space $X$ acts on a homogeneous element $x \in X$ by $\omega_X(x) = (-1)^{|x|} x$, where $|x| \in \{0,1\}$ is the $\Zb_2$-degree of $x$.

In \cite[Sect.\,6.1]{Arike:2011ab}, the subalgebra $\Ec_0$ is considered (called $P$ there), and it is shown in \cite[Sect.\,6.2]{Arike:2011ab} that $\widehat{\Lambda(\h)}$ is a projective $\Ec_0$-module. Furthermore, $\widehat T$ is free as an $\Ec_1$-module. Altogether:

\begin{lemma}[{\cite{Arike:2011ab}}]\label{lem:G-proj-E_SF}
$\Gc$ is projective as an $\Ec$-module.
\end{lemma}

If we assume the conjectures below, then $\Gc$ is a projective generator and the above lemma would also follow from Proposition \ref{prop:G-EndG-is-proj}. But for now let us proceed without making conjectures.

\medskip

For later use we fix the $\Ec$-module maps 
	$\pi = \pi_0 \oplus \pi_1$ and $\iota=\iota_0\oplus\iota_1$ 
needed in \eqref{eq:trace-map-def} to define the trace map:
\be
	\Ec_0 \otimes U(\widehat\h_{<0}) \xrightarrow{\pi_0} \widehat{\Lambda(\h)} \xrightarrow{\iota_0} \Ec_0 \otimes U(\widehat\h_{<0})
	~~,\quad
	\Ec_1 \otimes U(\widehat\h_{\mathrm{tw},<0}) \xrightarrow{\pi_1} \widehat{T} \xrightarrow{\iota_1} \Ec_1 \otimes U(\widehat\h_{\mathrm{tw},<0}) \ .
\ee
To define the maps $\pi_j$, $\iota_j$, first
note that by Poincar\'e-Birkhoff-Witt and by the definition of induced modules we have
 $\widehat{\Lambda(\h)} = U(\widehat\h_{<0}) \ot \Lambda(\h)$ and $\widehat{T} = U(\widehat\h_{\mathrm{tw},<0})$. We set
\begin{align}
\pi_0(f \ot u) &= f(u \ot \oneL) \ ,  &
\iota_0(u \ot \lambda) &= (-1)^{|u||\lambda|} \tfrac12\big( \lambda \ot u + \K  
\,\omega_{\Ec_0}(\lambda)\ot\omega_{U(\widehat\h_{<0})}(u)\big) \ ,
\nonumber\\
\pi_1(f \ot u) &= f(u) \ , &
\iota_1(x) &= \tfrac12\big(e_T \ot x + \K e_T \ot \omega_{U(\widehat\h_{\mathrm{tw},<0})}(x)\big) \ .
\label{eq:pi-iota-for-SF}
\end{align}
One verifies that $\pi_j, \iota_j$ are $\Ec_j$-module maps, and that $\pi_j \circ \iota_j = \id$.

\subsubsection*{Step 2}

Thanks to Lemma~\ref{lem:G-proj-E_SF} we can apply the theory of \cite{Arike:2011ab} and obtain the pseudo-trace function from Section \ref{sec:ptfun} for the $\VOA_\mathrm{ev}$-module $\Gc$ relative to $\Ec$. For a central form $\varphi \in C(\Ec)$, the pseudo-trace function is
\begin{align}
	\xi^\varphi_\Gc
	&= \mathrm{tr}_{U(\widehat\h_{<0})}\big(\, \big\{(\varphi|_{\Ec_0} \ot \id) \circ \iota_0\big\} \, o(v) \, e^{2 \pi i \tau (L_0 + N/12)}\circ  (\id_{U(\widehat\h_{<0})}\otimes \oneL)  \,\big)
	\nonumber\\
	&\qquad
	+ ~\mathrm{tr}_{U(\widehat\h_{\mathrm{tw},<0})}\big(\, \big\{(\varphi|_{\Ec_1} \ot \id) \circ \iota_1\big\} \, o(v) \, e^{2 \pi i \tau (L_0 + N/12)} \,\big) \ .
	\label{eq:pstr-for-SF}
\end{align}
and it is explicitly computed in Appendix \ref{app:Scalc}. From this explicit computation (see Lemma~\ref{lem:xi-inj-minus1-modes}) or from \cite[Thm.\,6.3.2]{Arike:2011ab} we get:

\begin{lemma}\label{lem:xiG-inj}
The map $\xi_\Gc : C(\Ec) \to C_1(\VOA_\mathrm{ev})$ in \eqref{eq:pstr-for-SF} is injective.
\end{lemma}

As pointed out in \cite[Cor.\,6.3.3]{Arike:2011ab}, for $N>1$ the pseudo-trace functions cannot be separated when just evaluating on the vacuum vector $v=\one$.
However, we show in Lemma~\ref{lem:xi-inj-minus1-modes} that restricting $v$ to products of modes $a_{-1}$, $a \in \h$, is sufficient.

\medskip

It will be technically convenient, but potentially confusing, to first pick an arbitrary central form $\eps \in C(\Ec)$ such that $\Ec$ becomes a symmetric Frobenius algebra. This form is different from $\delta$, which is one of the outcomes of step 2 and will be computed in due course.

Let $\{\alpha^i\,|\,i=1,\dots,2N\}$ be a symplectic basis of $\h$ with $(\alpha^{2n-1},\alpha^{2n})=1$, $n=1,\dots,N$. 
On $\Ec_0$ take $\eps$ to be non-vanishing 
only on elements in the top-degree of $\Lambda(\h)$ multiplied by the generator $\K\in\Cb\Zb_2$. 
On $\Ec_1$ it is non-zero only on the generator $\K \in \Cb\Zb_2$. The normalisation is 
\be\label{eq:SF-eps-norm}
	\eps( \alpha^1 \cdots \alpha^{2N} \K ) = 1
	\quad , \quad
	\eps( e_T \K ) = 1 \ .
\ee
The map $\eps : \Ec\to \Cb$ turns $\Ec$ into a symmetric Frobenius algebra \cite[Prop.\,6.1.2]{Arike:2011ab}. The centre of $\Ec$ is $Z(\Ec) = Z(\Ec_0) \oplus Z(\Ec_1)$ with (see \cite[Prop.\,6.1.3]{Arike:2011ab})
\be\label{eq:centerE}
	Z(\Ec_0) \,=\, \Lambda(\h)_\mathrm{ev} \,\oplus\, \Cb \alpha^1 \cdots \alpha^{2N} \K  \quad,\qquad Z(\Ec_1) = \Ec_1 \ .
\ee
From Lemma \ref{lem:ZA-CA-iso} we know that $\hat\eps : Z(\Ec) \to C(\Ec)$, $\hat\eps(z) = \eps(z \cdot (-))$ is an isomorphism. Instead of computing $\varphi_U \in C(\Ec)$, $U \in \Irr(\Rep(\VOA_\mathrm{ev}))$, we will give elements $c_U \in Z(\Ec)$ with $\hat\eps(c_U) = \varphi_U$. 
(We write $c_U$ here and reserve $\tildechi_U$ for the element satisfying $\hat\delta(\tildechi_U) = \varphi_U$ in Corollary \ref{cor:SF-irr-and-S} below.) 
We have
\begin{align}
c_{\one} &= \alpha^1 \cdots \alpha^{2N} \, (\K+1) \ ,
&
c_{T} &= e_T \, (\K+1) \ ,
\nonumber\\
c_{\Pi\one} &= \alpha^1 \cdots \alpha^{2N} \, (\K-1) \ ,
&
c_{\Pi T} &= e_T \, (\K-1) \ .
\label{eq:varphiU-SF}
\end{align}
This follows straightforwardly from the definition of $\varphi_U$ in \eqref{eq:varphiM-via-tr}.
Alternatively one may use \eqref{varphiU} to fix the $\varphi_U$, in which case \eqref{eq:varphiU-SF} can be obtained from \cite[Thm.\,6.3.2]{Arike:2011ab} or 
from direct calculation using \eqref{eq:pi-iota-for-SF} and \eqref{eq:pstr-for-SF} (see also the more explicit expressions \eqref{eq:pseudo-for-z123} and \eqref{eq:pseudo-for-ZLam}
and comments below them in Appendix \ref{app:Scalc}). 

\medskip

At this point we need to make a conjecture to proceed.

\begin{conjecture}\label{conj:SF-ZE=C1V}
The injective map $\xi_\Gc$ 
from Lemma~\ref{lem:xiG-inj} is an isomorphism.
\end{conjecture}

Under the above conjecture, in Appendix~\ref{app:Scalc} the modular 
$S$-transformation
 on pseudo-trace functions is computed, as well as the induced action $S_{\VOA_\mathrm{ev}}$ on $C(\Ec)$ as in \eqref{eq:SVOA-on-CE}.
Let $\tilde S_Z : Z(\Ec) \to Z(\Ec)$ be the transport of this
$S$-transformation
from $C(\Ec)$ to $Z(\Ec)$ via $\hat\eps$, i.e.\ $\hat\eps \circ \tilde S_Z = S_{\VOA_\mathrm{ev}} \circ \hat\eps$. 
(As with $c_U$ and $\phi_U$, we write $\tilde S_Z$ here and reserve $S_Z$ for the transported $S$-action computed with respect to $\hat\delta$ in Corollary \ref{cor:SF-irr-and-S} below.) 
To describe $\tilde S_Z$, we need some notation.

Let $\beta^1,\beta^2 \in \Cb^2$ be the standard basis. Identify $\Lambda(\Cb^2)^{\ot N} \cong \Lambda(\h)$ as super-algebras by sending $\beta^1$ in the $n$'th tensor factor to $\alpha^{2n-1}$ and $\beta^2$ to $\alpha^{2n}$. On $\Lambda(\Cb^2)$ let	$\sigma$ 
be the linear map whose representing matrix in the basis $( 1, \beta^1, \beta^2, \beta^2 \beta^1 )$ is
\be
\sigma = 
	\begin{pmatrix}
	0 & 0 & 0 & (-2 \pi)^{-1} \\
	0 & -i & 0 & 0 \\
	0 & 0 & -i & 0 \\
	-2 \pi & 0 & 0 & 0 \\
	\end{pmatrix}
\ee
	Note that $\sigma$ is parity-even.
To give $\tilde S_Z$ it is helpful to split $Z(\Ec)$ into sub-spaces as
\be\label{eq:ZE-splitting-Lam_P}
	Z(\Ec) = Z_\Lambda \oplus Z_P
	\quad , \quad
	Z_\Lambda = \Lambda(\h)_\mathrm{ev}
	~~,~~
	Z_P = \Cb \alpha^1 \cdots \alpha^{2N} \K
	\oplus \Ec_1 \ .
\ee
In Appendix \ref{app:Scalc} we prove:

\begin{lemma}\label{lem:tildeS-trans-SF}
	Assuming Conjecture \ref{conj:SF-ZE=C1V},
$\tilde S_Z$ is of the form 
$\tilde S_Z = \tilde S_{Z_\Lambda} \oplus \tilde S_{Z_P}$, where
\be
	\tilde S_{Z_\Lambda} : Z_\Lambda \to Z_\Lambda
	~~,\quad
	\tilde S_{Z_\Lambda} = \bigl(\sigma \otimes \cdots \otimes \sigma\bigr)\big|_{Z_{\Lambda}}
	\quad
	\text{($N$ factors)} \ ,
\ee
and, in the basis $( \tfrac12(c_{\one} + c_{\Pi\one}) , c_T , c_{\Pi T} )$,
\be\label{eq:tildeSP-mat}
	\tilde S_{Z_P} = 
	\begin{pmatrix}
	0 & 2^N & -2^N \\
	2^{-N-1} & 2^{-1} & 2^{-1} \\
	-2^{-N-1} & 2^{-1} & 2^{-1}
	\end{pmatrix} \qquad .
\ee
\end{lemma}

From this, we see that $\tilde S_Z(c_\one) = (2 \pi)^{-N} 
1_{\Lambda} + 2^{-N}e_T$. By definition, $\delta \in C(\Ec)$
 is given by 
\be
\delta = S_{\VOA_\mathrm{ev}}(\varphi_\one) = \hat\eps(\tilde S_Z(c_\one)) \ .
\ee 
Explicitly, $\delta$ has the same kernel as $\eps$ and the normalisation of $\delta$ is given by 
\be
	\delta( \alpha^1 \cdots \alpha^{2N} \K ) = (2 \pi)^{-N}
	\quad , \quad
	\delta( e_T \K ) = 2^{-N} \ .
\ee
Note that $\delta$ indeed induces a non-degenerate pairing on $\Ec$. Let us collect the remaining output of step 2. To do so, we use that 
$\hat\delta^{-1} ( \hat\eps(z)) = \big((2 \pi)^{N} 1_{\Lambda} + 2^{N} e_T \big) \cdot z$.

\begin{corollary}\label{cor:SF-irr-and-S}
	Assuming Conjecture \ref{conj:SF-ZE=C1V},
the central forms $\varphi_U$, $U\in\Irr(\Rep(\VOA_\mathrm{ev}))$ are given by
 $\varphi_U = \hat\delta(\tildechi_U)$ with
\begin{align}
 \tildechi_{\one} &= (2\pi)^N \alpha^1 \cdots \alpha^{2N} \, (\K+1) \ ,
&
 \tildechi_{T} &= 2^N \,e_T \, (\K+1) \ ,
\nonumber\\
 \tildechi_{\Pi\one} &= (2\pi)^N \alpha^1 \cdots \alpha^{2N} \, (\K-1) \ ,
&
 \tildechi_{\Pi T} &= 2^N \,e_T \, (\K-1) \ .
\end{align}
The modular $S$-transformation 
$S_Z$ on $Z(\Ec)$ is determined by that on $C(\Ec)$ via
$S_Z = \hat\delta^{-1} \circ S_{\VOA_\mathrm{ev}} \circ \hat\delta$. 
The linear map $S_Z$ decomposes as
$S_Z = S_{Z_\Lambda} \oplus S_{Z_P}$, where $S_{Z_\Lambda} = \tilde S_{Z_\Lambda}$ and $S_{Z_P}$ is represented by the same matrix as in \eqref{eq:tildeSP-mat}, but now with respect to the basis $( \tfrac12( \tildechi_{\one} +  \tildechi_{\Pi\one}) ,  \tildechi_T ,  \tildechi_{\Pi T} )$.
\end{corollary}

\subsubsection*{Step 3}

Since we already transported the calculation to 
	$Z(\Ec)$ in step 2, 
there is nothing to do in step 3.

\subsubsection*{Step 4}

In $Z(\Ec)$, the equation \eqref{eq:compute-NABC} for the structure constants reads
(note that via the isomorphism $\hat\delta$
the multiplication~`$\muldiag$' on $C(\Ec)$ corresponds to the
multiplication on $Z(\Ec)$)
\be\label{eq:NABC-def-SF}
	S_Z^{-1}\big( \, S_Z(\tildechi_A) \cdot S_Z(\tildechi_B) \, \big)
	~= \hspace{-1em}
	\sum_{C \in \Irr(\Rep(\VOA_\mathrm{ev}))} \hspace{-1em} N_{AB}^{~C} ~ \tildechi_C \ .
\ee
We compute
\begin{align}
S_Z(\tildechi_\one) &= 1_\Lambda + e_T = 
1_{\Ec} \ ,
&
S_Z(\tildechi_T) &= 2^{N-1}(\tildechi_\one + \tildechi_{\Pi\one}) + 2^N \K e_T \ ,
\nonumber\\
S_Z(\tildechi_{\Pi\one}) &= -1_\Lambda + e_T \ ,
&
S_Z(\tildechi_{\Pi T}) &= -2^{N-1}(\tildechi_\one + \tildechi_{\Pi\one}) + 2^N \K e_T \ ,
\end{align}
The structure constants $N_{AB}^{~C}$ in \eqref{eq:NABC-def-SF} then result in the following products of generators in $\Gr(\Rep(\VOA_\mathrm{ev}))$:
\begin{align}
&[\Pi\one] \ast [\Pi\one] = [\one]
~~,\quad
[\Pi\one] \ast [T] = [\Pi T]
~~,\quad
[\Pi\one] \ast [\Pi T] = [T] \ ,
\nonumber \\
&[T] \ast [T] = [T] \ast [\Pi T] = [\Pi T] \ast [T] = [\Pi T] \ast [\Pi T] 
	= 2^{2N-1} ([\one]+[\Pi\one]) \ .
\label{eq:Gr(SF)-product}
\end{align}
Here we use the notation `$\ast$' for the product computed from $S_Z$.
We would now like to apply  Theorem~\ref{thm:main} to conclude that the product `$\ast$' 
agrees with the usual product on $\Gr(\Rep(\VOA_\mathrm{ev}))$ induced from the tensor product. 
	This requires one more conjecture.

The braided monoidal structure on $\Rep(\VOA_\mathrm{ev})$ has so far not been computed in the VOA setting of \cite{Huang:2010}. However, in \cite{Davydov:2012xg,Runkel:2012cf} the structure of a finite ribbon category was determined on $\SF(\h)$ via a conformal block calculation for symplectic fermions. Furthermore, by Theorem \ref{thm:factequiv} and \cite[Prop.\,5.3]{Davydov:2012xg}, $\SF(\h)$ is factorisable.
We need the following stronger version of Corollary~\ref{cor:faith-fun_SF-RepVev} (see \cite[Conj.\,7.4]{Davydov:2016euo} for a more precise formulation).

\begin{conjecture}\label{conj:SF-RepV}
$(\widehat{-})_\mathrm{ev} : \SF(\h) \xrightarrow{\sim} \Rep( \VOA_\mathrm{ev})$ is an equivalence of $\Cb$-linear ribbon categories.
\end{conjecture}

For $N=1$, the structure of $\Rep(\VOA_\mathrm{ev})$ as a $\Cb$-linear category follows as the $p=2$ case of the relation between $W_{1,p}$-triplet models and a certain quantum 
group\footnote{The quantum group in question is the centraliser for a certain (non-minimum) choice of a projective generator in $\Rep( \VOA_\mathrm{ev})$. }
\cite{Feigin:2005xs,Nagatomo:2009xp}.
From this one can see that for $N=1$, the functor $(\widehat{-})_\mathrm{ev}$ is at least a $\Cb$-linear equivalence.

We can now apply Theorem~\ref{thm:main}:

\begin{corollary}
Under Conjectures~\ref{conj:SF-ZE=C1V} and \ref{conj:SF-RepV}, the expressions in 
\eqref{eq:Gr(SF)-product} describe the product in $\Gr(\Rep(\VOA_\mathrm{ev}))$.
\end{corollary}

\begin{proof}
We need to verify the three conjectures assumed in Theorem~\ref{thm:main}. \
Conjecture~\ref{conj:fact} is clear from Conjecture~\ref{conj:SF-RepV} 
as  $\SF(\h)$ is factorisable.
 From the latter we furthermore conclude that $\Gc$ is projective and that $\mathcal{E} = \End_{\VOA_\mathrm{ev}}(\Gc)$, rather than just being a subalgebra.
Conjecture~\ref{conj:C(End)-C1(V)-iso} now follows from  Conjecture~\ref{conj:SF-ZE=C1V}. 

As for Conjecture~\ref{conj:S-agrees}, part 1 holds as $\delta$ was seen explicitly to be non-degenerate above. 
Part 2 follows from a longer calculation carried out in~\cite{Gainutdinov:2015lja,FGR1,FGR2}. 
There, the $SL(2,\Zb)$-action on $\End(Id_{\SF})$ is computed by realising $\SF(\h)$ as representations of a factorisable ribbon quasi-Hopf algebra 
$\mathsf{Q}(N)$ ($N = \dim\h/2$), 
see \cite{Gainutdinov:2015lja,FGR2}. It is shown in \cite[Sec.\,6]{FGR2} that $\hat\delta$ makes the diagram \eqref{eq:S-compatible} commute and maps $\phi_U$ to $\varphi_U$, as required.
\end{proof}

Of course, given Conjecture~\ref{conj:SF-RepV} one can just use the explicit tensor product on $\SF(\h)$ to check that \eqref{eq:Gr(SF)-product} gives the product on $\Gr(\SF(\h)) \cong \Gr(\Rep(\VOA_\mathrm{ev}))$.
The main contents of the above corollary is thus the result of \cite{FGR2} that the modular and the categorical $S$-transformation agree, i.e.\ that \eqref{eq:S-compatible} commutes.

\begin{remark}
\label{rem:SF-example}
In \cite[Thm.\,5.7]{Abe:2011ab}, the dimensions 
$\Hom_{\VOA_\mathrm{ev}}(A \otimes B,C)$ are given for $A,B,C$ chosen from the four irreducible $\VOA_\mathrm{ev}$-modules by computing the dimension of the space of intertwining operators.\footnote{In \cite[Rem.\,5.8]{Abe:2011ab} fusion rules are described, but these are not the fusion rules of the Grothendieck ring.}
If we label the four irreducibles by elements of $\Zb_2 \times \Zb_2$ as
\be\label{eq:SF-irred-Z2Z2}
	M_{(0,0)} = (\widehat\one)_\mathrm{ev}
	~~,~~
	M_{(1,0)} = (\widehat{\Pi\one})_\mathrm{ev}
	~~,~~
	M_{(0,1)} = (\widehat T)_\mathrm{ev}
	~~,~~
	M_{(1,1)} = (\widehat{\Pi T})_\mathrm{ev} \ ,
\ee
then the result is that $\Hom_{\VOA_\mathrm{ev}}(M_g \otimes M_h, M_k)$ is one-dimensional if $g+h=k$ and zero otherwise.

If we {\em assume} that $\Rep(\VOA_\mathrm{ev})$ is rigid, that $(\widehat{\Pi\one})_\mathrm{ev}$ is $\ot$-invertible
and that $(\widehat T)_\mathrm{ev}$ and $(\widehat{\Pi T})_\mathrm{ev}$ are projective 
the result of \cite{Abe:2011ab} determines the tensor product of simple objects. 
	Indeed, any tensor product with $M_{(0,1)}$ and $M_{(1,1)}$ is then projective, too (see e.g.\ \cite[Prop.\,2.1]{Etingof:2003}). E.g.\ $M_{(1,0)} \ot M_{(0,1)} \cong \sum_{U \in \Irr} n_U P_U$ with $n_U = \dim\Hom_{\VOA_\mathrm{ev}}(M_{(1,0)} \ot M_{(0,1)},U)$ (as $P_U$ is the projective cover of $U$). In this way one finds
\be
(\widehat{\Pi\one})_\mathrm{ev} \ot (\widehat T)_\mathrm{ev} \cong (\widehat{\Pi T})_\mathrm{ev}
\quad , \quad
(\widehat T)_\mathrm{ev} \ot (\widehat T)_\mathrm{ev} \cong 
(\widehat{\Lambda(\h)})_\mathrm{ev} \ ,
\ee
which produces \eqref{eq:Gr(SF)-product} on the Grothendieck ring,
and also agrees with Conjecture~\ref{conj:SF-RepV}.
\end{remark}

\appendix

\section[Appendix: Symplectic fermion pseudo-trace functions]{Appendix:\\Symplectic fermion pseudo-trace functions}\label{app:Scalc}

\noindent
{\bf Virasoro action:} Let $\alpha^1,\dots,\alpha^{2N}$ be a symplectic basis of $\h$ as in Section \ref{sec:sf}. 
The action of the Virasoro generators on $M \in 
\Rep^\mathrm{fin}_{\flat,1}(\widehat\h_\mathrm{(tw)})$ is given by
(see e.g.\ \cite[Rem.\,2.5\,\&\,2.7]{Runkel:2012cf}),
\be
	\text{for $m \neq 0$} ~:\quad
	L_m = \sum_{k \in \Zb+\delta} \sum_{j=1}^N \alpha^{2j}_k \alpha^{2j-1}_{m-k} \ ,
\ee
where $\delta=0$ for $M$ untwisted and $\delta = \frac12$ for $M$ twisted. 
Before we give $L_0$, we introduce some new notation needed later. Let
$v \in \VOA_\mathrm{ev}$ be of the form
\be\label{eq:v-minus1-modes}
	v = h^1_{-1} \cdots h^{R}_{-1} \one
	\quad , \quad h^j \in \h ~~,~~R \ge 0 \ .
\ee
For such $v$ we define the endomorphism $J(v)$ of $M$ (untwisted case) and $J_\mathrm{tw}(v)$ (twisted case) as
\be
R=0 ~:~ J_\mathrm{(tw)}(\one) = \id
\quad , \quad
R>0 ~:~ 	J_\mathrm{(tw)}(v) ~= 
\hspace{-2em}
\sum_{m_1,\dots,m_R \in \Zb + \delta \,,~\sum_j m_j=0} 
\hspace{-2em}
: h^1_{m_1} \cdots h^R_{m_R} : \ ,
\ee
where, as before, $\delta=0$ for $M \in \Rep^\mathrm{fin}_{\flat,1}(\widehat\h)$ and $\delta=\frac12$ for $M \in \Rep^\mathrm{fin}_{\flat,1}(\widehat\h_\mathrm{tw})$. We will abbreviate, for $j = 1,\dots,N$,
\be\label{eq:gamma-tilde}
	\gamma_j = \alpha^{2j} \alpha^{2j-1} \in \Lambda(\h)
	\quad , \quad
	\tilde\gamma_j = \alpha^{2j}_{-1} \alpha^{2j-1}_{-1} \in U(\widehat \h) \ .
\ee
Then
\be\label{eq:L0-untw-tw}
\text{untwisted :} ~~
L_0 = \sum_{j=1}^N J(\tilde\gamma_j\one)
\quad , \qquad
\text{twisted :}~~
L_0 = - \tfrac{N}8+\sum_{j=1}^N J_\mathrm{tw}(\tilde\gamma_j\one)  \ .
\ee

\medskip

\noindent
{\bf Zhu's grading:} For the homogeneous transformation property of torus one-point functions in \eqref{eq:modxfer}, 
we need to insert homogeneous elements with respect to Zhu's grading. This grading arises from transforming the canonical local coordinates on the torus to the annulus, see \cite[Sect.\,4.2]{Zhu1996}. 
The resulting grading $\VOA = \bigoplus_{n \in \Zb_{\ge 0}} \VOA_{[n]}$ is by eigenspaces of the operator $L[0]$ given by
\be
	L[0] = L_0 + \tfrac12 L_1 - \tfrac16 L_2 + \tfrac1{12} L_3 + \text{(higher $L_m$'s)} \ .
\ee
On vectors of the form $v$ as in \eqref{eq:v-minus1-modes}, all $L_m$ with $m \ge 0$, except $L_0$ and $L_2$, act as zero. One checks:
\be
	L[0] \, e^{-\tfrac1{12}L_2} \,v ~=~ e^{-\tfrac1{12}L_2} \,L_0 \, v 
	~=~ R \, e^{-\tfrac1{12}L_2} \,v \ .
\ee
(Replacing $v$ by $\exp(-\tfrac1{12}L_2)v$ compensates the coordinate change from $z$ to $e^z-1$, but we will not need this fact.) Thus
\be\label{eq:v-homog}
	e^{-\tfrac1{12}L_2} \,v ~ \in ~ (\VOA_\mathrm{ev})_{[R]} \ .
\ee

\medskip

\noindent
{\bf Untwisted and twisted zero modes:}
We will compute the pseudo-trace functions for insertions of homogeneous vectors of degree $R$ of the form \eqref{eq:v-homog}. As in \cite[Sect.\,3.1]{Abe:2005}, one finds in the untwisted case: 
\be\label{eq:oJ}
	\text{for $M \in \Rep^\mathrm{fin}_{\flat,1}(\widehat\h)$}~:\quad
	o(e^{-\tfrac1{12}L_2} \,v) = J(e^{-\tfrac1{12}L_2} \,v) \ .
\ee
(Note that $L_2^k v$ is of the form~\eqref{eq:v-minus1-modes} and thus~\eqref{eq:oJ} is well-defined.) For the twisted case, we need to account for the transformation $\Delta(z)$ in \cite[Sect.\,4.1]{Abe:2005}. The relevant term is $\Delta(z) = \tfrac18 \sum_{k=1}^N \alpha^{2k}_1 \alpha^{2k-1}_1 z^{-2} + \text{(other terms)}$. Then, from \cite[Sect.\,4.1]{Abe:2005},
\be
	\text{for $M \in \Rep^\mathrm{fin}_{\flat,1}(\widehat\h_\mathrm{tw})$}
	~:\quad
	o(e^{-\tfrac1{12}L_2} \,v) = 
	J_{\mathrm{tw}}(e^{(\tfrac18-\tfrac1{12})L_2} \,v) \ .
\ee

\medskip

\noindent
{\bf Pseudo-trace functions for untwisted modules:} To avoid spelling out $e^{-\tfrac1{12}L_2}$ in every argument, we set
\be
	\zeta^\varphi_\Gc(v,\tau) ~:=~ \xi^\varphi_\Gc\big(e^{-\tfrac1{12}L_2} \,v,\tau\big) \ .
\ee
We start by evaluating  the first summand in \eqref{eq:pstr-for-SF} for an insertion of the form \eqref{eq:v-homog}:
\be
	\zeta^\varphi_0(v,\tau) ~:=~
	\mathrm{tr}_{U(\widehat\h_{<0})}\big(\, \big\{(\varphi|_{\Ec_0} \ot \id) \circ \iota_0\big\} \, o(e^{-\tfrac1{12}L_2} \,v) \, e^{2 \pi i \tau (L_0 + N/12)} \circ  (\id_{U(\widehat\h_{<0})}\otimes \oneL) \,\big) \ .
\ee
	Since central forms $\varphi$ on $\Ec_0$ are non-vanishing only in even
 degree\footnote{
Recall that with the central form $\eps$ from~\eqref{eq:SF-eps-norm}, $\Ec_0$ is  a symmetric Frobenius algebra. The form $\eps$ is non-vanishing only in top degree, which is even.
By Lemma \ref{lem:ZA-CA-iso}, all other central forms are of the form $\eps(z \cdot -)$, where $z$ is a central element of $\Ec_0$. But all elements of the centre $Z(\Ec_0)$ are of even degree, see~\eqref{eq:centerE}. Hence all central forms vanish in odd degrees.},
we have, for $u \ot \lambda \in \widehat{\Lambda(\h)} = U(\widehat\h_{<0}) \ot \Lambda(\h)$ (recall \eqref{eq:pi-iota-for-SF})
\be
\big((\varphi|_{\Ec_0} \ot \id) \circ \iota_0\big)(u \ot \lambda)
~=~ \tfrac12 \varphi|_{\Ec_0}(\lambda) \,u + \tfrac12 \varphi|_{\Ec_0}(\K\lambda) \, \omega_{U(\widehat\h_{<0})}(u) \ .
\ee

Consider $w \in \VOA_\mathrm{ev}$ of the form
\be\label{eq:w-insertion-pt}
	w = \tilde\gamma_{l_1} \cdots \tilde\gamma_{l_r} \alpha^{j_1}_{-1} \cdots \alpha^{j_s}_{-1} \one ~\in \VOA_\mathrm{ev} \ ,
\ee
where $\tilde\gamma_{l}$, for  $l = 1,\dots,N$, were introduced in~\eqref{eq:gamma-tilde} and
 $\{j_1,\dots,j_s\}$ contains no pair $\{2l-1,2l\}$ for $l=1,\dots,N$. 
One can convince oneself that under the trace one can replace
\be\label{eq:under-tr}
	J(w) ~\leadsto~
	J(\tilde\gamma_{l_1}\one) \cdots J(\tilde\gamma_{l_r}\one) \, \alpha^{j_1}_{0} \cdots \alpha^{j_s}_{0} \ .
\ee
Furthermore, one verifies that
\be\label{eq:elamL2}
	e^{\lambda L_2} \, w ~=~
	(\tilde\gamma_{l_1}-\lambda) \cdots (\tilde\gamma_{l_r}-\lambda)\, \alpha^{j_1}_{-1} \cdots \alpha^{j_s}_{-1} \one \ .
\ee
Combining this with \eqref{eq:L0-untw-tw} in the form $L_0 + \tfrac{N}{12} = \sum_{j=1}^N J((\tilde\gamma_j+\tfrac1{12})\one)$ we get
\begin{align}
	\zeta^\varphi_0(w,\tau) ~=~& (2 \pi i)^{-r} \tfrac{\partial}{\partial\tau_{l_1}} \cdots \tfrac{\partial}{\partial\tau_{l_r}}
	 \mathrm{tr}_{U(\widehat\h_{<0})}\big(\, \big\{(\varphi|_{\Ec_0} \ot \id) \circ \iota_0\big\} \, \alpha^{j_1}_{0} \cdots \alpha^{j_s}_{0}
	\nonumber \\
&	\qquad \times e^{2 \pi i \sum_{j=1}^N  J((\tilde\gamma_j+\tfrac1{12})\one)\tau_j } \circ  (\id_{U(\widehat\h_{<0})}\otimes \oneL)\,\big)\big|_{\tau_j=\tau} \ .
\end{align}
Next we note that $\sum_{j=1}^N J((\tilde\gamma_j+\tfrac1{12})\one)$ acts on 
$U(\widehat\h_{<0}) \ot \Lambda(\h)$ as
\be
	\sum_{j=1}^NJ((\tilde\gamma_j+\tfrac1{12})\one) 
	~=~ \sum_{j=1}^N \id \ot \gamma_j \,+\, (D+\tfrac1{12}) \ot \id \ ,
\ee
where $D$ is the $\Zb_{\ge 0}$-valued 
	(negative of the)
degree in $U(\widehat\h_{<0})$.
Define $\widehat\h^{N=1}$ to be the mode algebra for a single pair of symplectic fermions. We will need the characters \cite{Kausch:1995py,Gaberdiel:1996np,Abe:2005} (as usual, $q=e^{2 \pi i \tau}$)
\be\label{eq:ch-N=1-ns}
\chi^{N=1}_{ns,\pm}(\tau) = \mathrm{tr}_{U(\widehat\h^{N=1}_{<0})}\big(\, P_\pm \,
e^{2 \pi i \tau (D + 1/12)} \,\big)	
= \Big( q^{\frac1{24}} \prod_{n=1}^\infty (1\pm q^n) \Big)^2 \ ,
\ee
where $P_+ = \id$ while $P_- = \omega$ is the parity involution, and where $D$ denotes the $\Zb_{\ge 0}$-valued degree on $U(\widehat\h^{N=1}_{<0})$.

Combining all this, we arrive at the final expression
\begin{align}
\zeta^\varphi_0(w,\tau) ~=~ &
\tfrac12 \,(2 \pi i)^{-r}
\tfrac{\partial}{\partial\tau_{l_1}} \cdots \tfrac{\partial}{\partial\tau_{l_r}}
\Big\{ 
\varphi|_{\Ec_0}\big(\alpha^{j_1} \cdots \alpha^{j_s} 
\, e^{2 \pi i \sum_{j=1}^N \gamma_j \tau_j} \big)\,
\prod_{j=1}^N \chi^{N=1}_{ns,+}(\tau_j)
\nonumber\\
& \qquad +~ 
\varphi|_{\Ec_0}\big(\K\alpha^{j_1} \cdots \alpha^{j_s} \, e^{2 \pi i \sum_{j=1}^N \gamma_j \tau_j}\big)\,
\prod_{j=1}^N \chi^{N=1}_{ns,-}(\tau_j)
 \Big\}\Big|_{\tau_j = \tau} \ ,
 \label{eq:xi0-final}
\end{align}
where we used that $U(\widehat\h_{<0}) \cong \bigotimes_{j=1}^{N} U(\widehat\h^{N=1}_{<0})$.

\medskip

\noindent
{\bf Pseudo-trace functions for twisted modules:}
Next we compute the second summand in \eqref{eq:pstr-for-SF} for an insertion of the form $e^{-L_2/12} \,w$ with $w$ as in \eqref{eq:w-insertion-pt}:
\be
	\zeta^\varphi_1(w,\tau) ~:=~
 \mathrm{tr}_{U(\widehat\h_{\mathrm{tw},<0})}\big(\, \big\{(\varphi|_{\Ec_1} \ot \id) \circ \iota_1\big\} \, o(e^{-\tfrac1{12}L_2} \,w) \, e^{2 \pi i \tau (L_0 + N/12)} \,\big) \ .
\ee
The calculation is similar to the untwisted case, but easier. One has to use \eqref{eq:elamL2} for $\lambda = \tfrac18-\tfrac1{12}$ and \eqref{eq:L0-untw-tw} in the form $L_0 + \tfrac{N}{12} = \sum_{j=1}^N J_\mathrm{tw}((\tilde\gamma_j-\tfrac18+\tfrac1{12})\one)$.
Note that the analogue of the statement in~\eqref{eq:under-tr} is
$	J_{\mathrm{tw}}(w) \leadsto
	\delta_{s,0}J_{\mathrm{tw}}(\tilde\gamma_{l_1}\one) \cdots J_{\mathrm{tw}}(\tilde\gamma_{l_r}\one)$
	and that $\sum_{j=1}^N J_{\mathrm{tw}}(\tilde\gamma_{j}\one)$ acts as $D$ on $\widehat{T}$.
 We need the characters \cite{Kausch:1995py,Gaberdiel:1996np,Abe:2005}, for $P_\pm$ as in \eqref{eq:ch-N=1-ns},
\be\label{eq:ch-N=1-r}
\chi^{N=1}_{r,\pm}(\tau)
 = \mathrm{tr}_{U(\widehat\h^{N=1}_{\mathrm{tw},<0})}\big(\, P_\pm \,
e^{2 \pi i \tau (D 
	- \frac18
+ \frac1{12})} \,\big)	
 =  \Big( q^{-\frac1{48}} \prod_{n=1}^\infty (1\pm q^{n-\frac12}) \Big)^2 \ .
\ee
In this way, one finally obtains
\begin{align}
\zeta^\varphi_1(w,\tau) ~=~ &
\tfrac12 \, \delta_{s,0} \,(2 \pi i)^{-r}
\tfrac{\partial}{\partial\tau_{l_1}} \cdots \tfrac{\partial}{\partial\tau_{l_r}}
\Big\{ 
\varphi|_{\Ec_1}(e_T)\,
\prod_{j=1}^N \chi^{N=1}_{r,+}(\tau_j)
\nonumber\\
& \qquad +~ 
\varphi|_{\Ec_1}(\K e_T)\,
\prod_{j=1}^N \chi^{N=1}_{r,-}(\tau_j)
 \Big\}\Big|_{\tau_j = \tau} \ .
 \label{eq:xi1-final}
\end{align}

It is shown in \cite[Thm.\,6.3.2]{Arike:2011ab} that the map $C(\Ec_0) \to C_1(\VOA)$, 
$\varphi \mapsto \zeta^\varphi_0$ is injective. The explicit calculation above allows one to make a slightly sharper statement.

\begin{lemma}\label{lem:xi-inj-minus1-modes}
The subspace of $\VOA_\mathrm{ev}$ spanned by all $w$ of the form \eqref{eq:w-insertion-pt} is sufficient to separate all $\zeta^\varphi_\Gc$. That is, if $\zeta^\varphi_\Gc(w,\tau)=0$ for all $w$ of the form \eqref{eq:w-insertion-pt} and all $\tau \in \mathbb{H}$, then $\varphi=0$.
\end{lemma}

\begin{proof}
Suppose  $\zeta^\varphi_\Gc(w,\tau)=0$ for all $w$ of the form \eqref{eq:w-insertion-pt} and all $\tau \in \mathbb{H}$.

Note that the functions $\tau^a \,(\frac{\partial}{\partial \tau}\chi^{N=1}_{ns,-}(\tau))^b\, \chi^{N=1}_{ns,-}(\tau)^{N-b}$, for $a \in \Zb_{\ge 0}$ and $b=0,1,\dots,N$, are linearly independent. 

From \eqref{eq:xi1-final} we see that $\varphi|_{\Ec_1}(e_T) = 0$ and
$\varphi|_{\Ec_1}(\K e_T) = 0$, and thus $\varphi|_{\Ec_1}=0$.
In \eqref{eq:xi0-final}, consider all terms where no derivative acts on the characters and use that $\exp(2 \pi i \gamma_j \tau_j) = 1 + 2 \pi i \gamma_j \tau_j$ to see, for $\nu \in \{0,1\}$,
\begin{align}
 \varphi|_{\Ec_0}\Big(\K^\nu \alpha^{j_1} \cdots \alpha^{j_s} 
 \hspace{-1em} \prod_{j\in\{l_1,\dots,l_r\}} \hspace{-.8em}
 (2 \pi i \gamma_j)
 \hspace{-.5em} \prod_{j\notin\{l_1,\dots,l_r\}}  \hspace{-.5em}(1 + 2 \pi i \gamma_j \tau) \Big)
= 0 \ .
\end{align}
As this holds for all $r,s$, $j_1,\dots,j_s$, $l_1,\dots,l_r$ and $\tau$, we conclude $\varphi|_{\Ec_0} = 0$.
\end{proof}

\medskip

\noindent
{\bf Modular $S$-transformation:}
The behaviour of the characters \eqref{eq:ch-N=1-ns} and \eqref{eq:ch-N=1-r} under $S$-transformation is \cite{Kausch:1995py,Gaberdiel:1996np,Abe:2005}
\begin{align}
\chi^{N=1}_{ns,+}(\tfrac{-1}\tau) &\,=\, \tfrac12 \,\chi^{N=1}_{r,-}(\tau) \ ,
&
\chi^{N=1}_{r,+}(\tfrac{-1}\tau) &\,=\, \chi^{N=1}_{r,+}(\tau) \ ,
\nonumber\\
\chi^{N=1}_{ns,-}(\tfrac{-1}\tau) &\,=\, -i \tau \,\chi^{N=1}_{ns,-}(\tau) \ ,
&
\chi^{N=1}_{r,-}(\tfrac{-1}\tau) &\,=\, 2 \, \chi^{N=1}_{ns,+}(\tau) \ .
\label{eq:N=1-Strans}
\end{align}

We use the isomorphism $Z(\Ec)\to C(\Ec)$ given by $\hat\eps(z) = \eps(z \cdot (-))$ with $\eps$ as in \eqref{eq:SF-eps-norm}. We split $Z(\Ec) = Z_\Lambda \oplus Z_P$ as in \eqref{eq:ZE-splitting-Lam_P}. The basis for $Z_P$ used in Lemma \ref{lem:tildeS-trans-SF} is
\be
	z_1 = \alpha^1 \cdots \alpha^{2N} \, \K
	~~,\quad
	z_2 = e_T \, (\K+1)
	~~,\quad
	z_3 = e_T \, (\K-1) \ .
\ee
Abbreviate $\zeta^{\hat\eps(z)}_\Gc =: \zeta^z_\Gc$. Then
\begin{align}
\zeta^{z_1}_\Gc(w,\tau)
~&=~ 
\tfrac12 \, \delta_{s,0} \,(2 \pi i)^{-r}
\tfrac{\partial}{\partial\tau_{l_1}} \cdots \tfrac{\partial}{\partial\tau_{l_r}}
\prod_{j=1}^N \chi^{N=1}_{ns,+}(\tau_j)
\Big|_{\tau_j = \tau} \ ,
\nonumber\\
\zeta^{z_2}_\Gc(w,\tau)
~&=~ 
\tfrac12 \, \delta_{s,0} \,(2 \pi i)^{-r}
\tfrac{\partial}{\partial\tau_{l_1}} \cdots \tfrac{\partial}{\partial\tau_{l_r}}
\Big\{ 
\prod_{j=1}^N \chi^{N=1}_{r,+}(\tau_j)
~+~ 
\prod_{j=1}^N \chi^{N=1}_{r,-}(\tau_j)
 \Big\}\Big|_{\tau_j = \tau} \ ,
\nonumber\\
\zeta^{z_3}_\Gc(w,\tau)
~&=~ 
\tfrac12 \, \delta_{s,0} \,(2 \pi i)^{-r}
\tfrac{\partial}{\partial\tau_{l_1}} \cdots \tfrac{\partial}{\partial\tau_{l_r}}
\Big\{ 
\prod_{j=1}^N \chi^{N=1}_{r,+}(\tau_j)
~-~ 
\prod_{j=1}^N \chi^{N=1}_{r,-}(\tau_j)
 \Big\}\Big|_{\tau_j = \tau} \ .
 \label{eq:pseudo-for-z123}
\end{align}
We note that $\zeta^{z_2}_\Gc(w,\tau)$  and $\zeta^{z_3}_\Gc(w,\tau)$ from~\eqref{eq:pseudo-for-z123} are equal to $\zeta^{\id^*}_{\widehat{T}_{\mathrm{ev}}}(w,\tau)$ and $\zeta^{\id^*}_{\widehat{T}_{\mathrm{odd}}}(w,\tau)$ correspondingly. So by definition~\eqref{varphiU}, we conclude that $\varphi_{T}=\eps(z_2 \cdot -)$ and $\varphi_{\Pi T}=\eps(z_3 \cdot -)$ which agrees with~\eqref{eq:varphiU-SF}, where we set $c_T = z_2$ and $c_{\Pi T}=z_3$.

Note that for $F(\tau) = \frac{\partial}{\partial \tau} f(\tau)$ we have $F(-1/\tau) = \tau^2 \frac{\partial}{\partial \tau}\big( f(-1/\tau)\big)$.
The $S$-transformation reads
\begin{align}
\tau^{-2r} \,\zeta^{z_1}_\Gc(w,\tfrac{-1}\tau)
~&=~   2^{-N-1} \,\big( \, \zeta^{z_2}_\Gc(w,\tau) - \zeta^{z_3}_\Gc(w,\tau) \, \big) \ ,
\nonumber\\
\tau^{-2r} \,\zeta^{z_2}_\Gc(w,\tfrac{-1}\tau)
~&=~  \tfrac12 \, \zeta^{z_2}_\Gc(w,\tau) + \tfrac12\, \zeta^{z_3}_\Gc(w,\tau) +2^N \,\zeta^{z_1}_\Gc(w,\tau) \ ,
\nonumber\\
\tau^{-2r} \,\zeta^{z_3}_\Gc(w,\tfrac{-1}\tau)
~&=~   \tfrac12 \,\zeta^{z_2}_\Gc(w,\tau) + \tfrac12\, \zeta^{z_3}_\Gc(w,\tau) -2^N \,\zeta^{z_1}_\Gc(w,\tau)  \ ,
\end{align}
where the factor $\tau^{-2r}$ is the prefactor in \eqref{eq:modxfer} for $w \in (\VOA_\mathrm{ev})_{[2r+s]}$ (the above 
	pseudo-trace functions
can be non-zero only for $s=0$). This is in agreement with $\tilde S_{Z_P}$ in 
Lemma~\ref{lem:tildeS-trans-SF}.

For $z \in Z_\Lambda$ the pseudo-trace function becomes
\begin{align}
\zeta^{z}_\Gc(w,\tau) ~=~ 
\tfrac12 \,(2 \pi i)^{-r}
\tfrac{\partial}{\partial\tau_{l_1}} \cdots \tfrac{\partial}{\partial\tau_{l_r}}
\Big\{ & \eps\big(z\K\alpha^{j_1} \cdots \alpha^{j_s} \, {\textstyle \prod_{j=1}^N} (1 + 2 \pi i \gamma_j \tau_j)\big)
\nonumber \\
& \qquad
\times
{\textstyle \prod_{j=1}^N} \, \chi^{N=1}_{ns,-}(\tau_j)
 \Big\}\Big|_{\tau_j = \tau} \ .
 \label{eq:pseudo-for-ZLam}
\end{align}
where we used that $\exp(2 \pi i \gamma_j \tau_j) = 1 + 2 \pi i \gamma_j \tau_j$.
We note that $\zeta^{c_{\one}}_\Gc(w,\tau)$  and $\zeta^{c_{\Pi\one}}_\Gc(w,\tau)$ for the central elements from~\eqref{eq:varphiU-SF} are equal to $\zeta^{\id^*}_{\VOA_{\mathrm{ev}}}(w,\tau)$ and $\zeta^{\id^*}_{\VOA_{\mathrm{odd}}}(w,\tau)$ correspondingly. We thus conclude that $\varphi_{\one}=\eps(c_{\one} \cdot -)$ and $\varphi_{\Pi\one}=\eps(c_{\Pi\one} \cdot -)$ which agrees with~\eqref{eq:varphiU-SF}.

The $S$-transformation of~\eqref{eq:pseudo-for-ZLam} is
\begin{align}
\tau^{-2r-s} \,\zeta^{z}_\Gc(w,\tfrac{-1}\tau)
~=~ \tau^{-s}
\tfrac12 \,(2 \pi i)^{-r}
\tfrac{\partial}{\partial\tau_{l_1}} \cdots \tfrac{\partial}{\partial\tau_{l_r}}
\Big\{ & \eps\big(z\K\alpha^{j_1} \cdots \alpha^{j_s} \,
 {\textstyle \prod_{j=1}^N} (- i \tau_j - 2 \pi  \gamma_j )
 \big)
\nonumber \\
& \qquad
\times
{\textstyle \prod_{j=1}^N} \, \chi^{N=1}_{ns,-}(\tau_j)
 \Big\}\Big|_{\tau_j = \tau} \ .
\end{align}
We now need to assume that Conjecture~\ref{conj:SF-ZE=C1V} holds (i.e.\ that 
$\xi_\Gc : C(\Ec) \xrightarrow{\sim} C_1(\VOA_\mathrm{ev})$).\footnote{
If we do not assume that 
$\xi_\Gc$
 is an isomorphism, then the calculation giving $y$ shows the identity \eqref{eq:SF-Strans-Z_Lam_ydef} for all $w$ of the form \eqref{eq:w-insertion-pt}. However, we cannot exclude that outside of such $w$ the two sides of \eqref{eq:SF-Strans-Z_Lam_ydef} start to differ. By Lemma~\ref{lem:xi-inj-minus1-modes} this can only happen if $\tau^{-2r-s} \,\zeta^{z}_\Gc(w,\tfrac{-1}\tau)$ is a torus 1-point function which is not in the image of $\xi_\Gc$.
 This possibility is excluded by Conjecture~\ref{conj:SF-ZE=C1V}.
}
Then by Theorem \ref{thm:C1-mod-inv} and Lemma \ref{lem:xi-inj-minus1-modes}, 
there exists a unique $y \in Z(\Ec)$ such that for all $w$ of the form \eqref{eq:w-insertion-pt} and $\tau \in \mathbb{H}$,
\be\label{eq:SF-Strans-Z_Lam_ydef}
	\zeta^{y}_\Gc(w,\tau) ~=~ \tau^{-2r-s} \,\zeta^{z}_\Gc(w,\tfrac{-1}\tau) \ .
\ee
We will derive a series of necessary conditions from this equation which will ultimately determine $y$ uniquely. By the above argument, this $y$ then automatically solves \eqref{eq:SF-Strans-Z_Lam_ydef}.

When comparing the coefficients of $ \chi^{N=1}_{ns,-}(\tau)^{N}$ in \eqref{eq:SF-Strans-Z_Lam_ydef}, i.e.\ the case where no $\tau$-derivatives act on the characters, one finds that \eqref{eq:SF-Strans-Z_Lam_ydef} implies
\begin{align}
&
 \tau^{s} \,\eps\Big(y\K\alpha^{j_1} \cdots \alpha^{j_s} 
 \hspace{-1em} \prod_{j\in\{l_1,\dots,l_r\}} \hspace{-.8em}
 (2 \pi i \gamma_j)
 \hspace{-.5em} \prod_{j\notin\{l_1,\dots,l_r\}}  \hspace{-.5em}(1 + 2 \pi i \gamma_j \tau) \Big)
 \nonumber\\
&\quad =~ 
 \eps\Big(z\K\alpha^{j_1} \cdots \alpha^{j_s} \,
   (- i)^r \,
  \hspace{-.6em}\prod_{j\notin\{l_1,\dots,l_r\}}\hspace{-.8em}
  (- i \tau - 2 \pi  \gamma_j)
  \Big)
 \ .
\end{align}
In this expression in turn, we compare the coefficient of $\tau^{s}$. 
To write the resulting identity, let $[k] \in \{1,\dots,N\}$ denote the pair that $k \in \{1,\dots,2N\}$ belongs to, i.e.\ $k \in \{ 2[k]-1,2[k] \}$. We obtain
\be\label{eq:y-nec-cond}
 \eps\Big(y\K\alpha^{j_1} \cdots \alpha^{j_s} \,
 \gamma_{l_1} \cdots \gamma_{l_r}
 \Big)
~=~ 
 \eps\Big(z\K\alpha^{j_1} \cdots \alpha^{j_s} \,
   (- i)^{s} \, (-2\pi)^{N-2r-s}
  \hspace{-2em}\prod_{j\notin\{l_1,\dots,l_r,[j_1],\dots,[j_s]\}}\hspace{-2em}
  \gamma_j
  \Big)
 \ .
\ee
Requiring this identity for all $r,s$ and $j_1,\dots,j_s$, $l_1,\dots,l_r$ determines $y$ uniquely. In particular, we see that necessarily $y \in Z_\Lambda$. To compute $y$ explicitly, let $z \in Z_\Lambda$ be given by
\be
	z = \gamma_{k_1} \cdots \gamma_{k_p} \alpha^{i_1} \cdots \alpha^{i_q} \ . \ee
To establish the formula for $\tilde S_{Z_\Lambda}$ in Lemma \ref{lem:tildeS-trans-SF}, one needs to verify that the unique $y$ corresponding to this $z$ is given by
\be
	y = (-2 \pi)^{N-2p-q}\, (-i)^q \, \alpha^{i_1} \cdots \alpha^{i_q}
	\hspace{-2em}\prod_{j\notin\{k_1,\dots,k_p,[i_1],\dots,[i_q]\}}\hspace{-2em}
  \gamma_j \ .
\ee
That the above pair of $z$ and $y$ indeed solves \eqref{eq:y-nec-cond} for all choices of $r,s$ and $j_1,\dots,j_s$, $l_1,\dots,l_r$ can now be checked by direct calculation.

\medskip

This completes the proof of Lemma \ref{lem:tildeS-trans-SF}.

\newcommand\arxiv[2]      {\href{http://arXiv.org/abs/#1}{#2}}
\newcommand\doi[2]        {\href{http://dx.doi.org/#1}{#2}}
\newcommand\httpurl[2]    {\href{http://#1}{#2}}

\end{document}